\documentclass[12pt]{amsart}

\usepackage{graphics}
\usepackage{amsmath}
\usepackage{amsfonts}
\usepackage{amssymb}
\usepackage{amscd}
\usepackage{mathrsfs}
\usepackage{enumerate}
\usepackage{amsthm}
\input xy
\xyoption{all}

\newcommand{\CC}{{\mathbb C}}

\newcommand{\RR}{{\mathbb R}}

\newcommand{\Cont}{{\mathscr C}}
\newcommand{\NN}{{\mathbb N}}
\newcommand{\OO}{{\mathscr O}}

\DeclareMathOperator{\Aut}{Aut}
\DeclareMathOperator{\id}{id}

\DeclareMathOperator{\dist}{dist}

\newtheorem{theorem}{\bf Theorem}
\newtheorem{lemma}{\bf Lemma}
\newtheorem{proposition}{\bf Proposition}
\newtheorem{corollary}{\bf Corollary}
\newtheorem{definition}{\bf Definition}

\newtheorem*{ntheorem}{\bf Main Theorem}
\newtheorem*{ncorollary}{\bf Corollary}

\begin{document}

\title{A strong Oka principle for embeddings\\ of some planar domains into $\CC\times\CC^*$}

\author{Tyson Ritter}
\address{Tyson Ritter, School of Mathematical Sciences, University of Adelaide, Adelaide SA 5005, Australia}
\email{tyson.ritter@adelaide.edu.au}


\date{16 November 2010.}


\begin{abstract}
Gromov, in his seminal 1989 paper on the Oka principle, introduced the notion of an elliptic manifold and proved that every continuous map from a Stein manifold to an elliptic manifold is homotopic to a holomorphic map. We show that a much stronger Oka principle holds in the special case of maps from certain open Riemann surfaces called circular domains into $\CC\times\CC^*$, namely that every continuous map is homotopic to a proper holomorphic embedding. An important ingredient is a generalisation to $\CC\times\CC^*$ of recent results of Wold and Forstneri\v c on the long-standing problem of properly embedding open Riemann surfaces into $\CC^2$, with an additional result on the homotopy class of the embeddings. We also give a complete solution to a question that arises naturally in L\'arusson's holomorphic homotopy theory, of the existence of acyclic embeddings of Riemann surfaces with abelian fundamental group into 2-dimensional elliptic Stein manifolds.
\end{abstract}

\maketitle

\section{Introduction}
\label{sec_introduction}
\noindent By a theorem of Remmert-Narasimhan-Bishop \cite{Bishop:1961,Narasimhan:1960,Remmert:1956}, every Stein manifold of dimension $n$ may be properly holomorphically embedded into $\CC^{2n+1}$. (Throughout this paper, all embeddings will be both proper and holomorphic.) Later refinements to this theorem by Eliashberg and Gromov \cite{Eliashberg:1992}, followed by Sch\"urmann \cite{Schurmann:1997}, show that every Stein manifold of dimension $n > 1$ can be embedded into $\CC^{[3n/2]+1}$, but the proof fails in the case $n = 1$. Whether every Stein manifold of dimension 1, that is, every open Riemann surface, can be embedded into $\CC^2$, is a long-standing and important unsolved question of complex geometry.

Kasahara and Nishino \cite{Stehle:1972} were the first to show that the open disc embeds into $\CC^2$, making use of proper open subsets of $\CC^2$ biholomorphic to $\CC^2$, known as Fatou-Bieberbach domains. Laufer \cite{Laufer:1973} also used Fatou-Bieberbach domains to show that all non-degenerate annuli embed into $\CC^2$, while Alexander \cite{Alexander:1977} gave an explicit embedding of both the disc and punctured disc into $\CC^2$ using the elliptic modular function. Globevnik and Stens\o nes \cite{Globevnik:1995} proved the more general result that every bounded, finitely connected planar domain without isolated boundary points embeds into $\CC^2$ (such domains are biholomorphic to what we call \emph{circular domains} in Section~\ref{sec_strongOka}). Recently, Wold \cite{Wold:2006,Wold:2006a} introduced a more general method that allowed him to embed a larger class of open Riemann surfaces into $\CC^2$, culminating in a result of Forstneri\v c and Wold \cite{Forstneric:2009} that states if a compact bordered Riemann surface can be embedded into $\CC^2$, then so too can the open Riemann surface which is its interior. By a \emph{bordered Riemann surface} we mean a two-dimensional smooth manifold with boundary, equipped with a complex structure, and whose boundary is thus a smooth one-dimensional manifold, namely a disjoint union of circles and lines.

A natural extension of the embedding problem for Stein manifolds, and open Riemann surfaces in particular, is to consider embeddings which are \emph{acyclic}, meaning that they induce a homotopy equivalence between their source and target. Of course it then becomes necessary to permit more general manifolds as targets, but we still wish for them to be similar (in some sense) to affine space. This leads to the concept of an elliptic complex manifold.

Elliptic manifolds are those complex manifolds $X$ that permit a holomorphic vector bundle $E\to X$ with a holomorphic map $s : E\to X$ (called a \emph{dominating spray}) such that for every $x \in X$, $s|_{E_x}$ maps $0 \in E_x$ to $x \in X$ and is a submersion at $0$. Elliptic manifolds were first introduced by Gromov in his study of the Oka principle \cite{Gromov:1989}. The Oka principle refers to a collection of results that state there are only topological obstructions to solving certain holomorphically defined problems involving Stein manifolds. In \cite{Gromov:1989}, Gromov proved the following result, sometimes referred to as the \emph{weak Oka property} for elliptic manifolds:

\begin{quote}
The inclusion of the space of holomorphic maps from a Stein manifold to an elliptic manifold into the space of continuous maps is a weak homotopy equivalence. In particular, every continuous map from a Stein manifold to an elliptic manifold is homotopic to a holomorphic map.
\end{quote}
 
This result, together with various other stronger Oka properties since proved in the literature (for example, see the recent survey \cite{Forstneric:2010}), indicate that elliptic manifolds can be thought of as having many maps into them from Stein manifolds. For this reason, elliptic manifolds are suitable targets for acyclic embeddings of Stein manifolds.

Despite the importance of elliptic manifolds, relatively few examples are known and little is understood about their topology in general. One by-product of investigating acyclic embeddings of Stein manifolds into elliptic manifolds is that the existence of such embeddings would help determine the possible homotopy types that elliptic manifolds can have. Indeed, if it were possible to embed every Stein manifold acyclically into an elliptic manifold, then because Stein manifolds are known to have all the homotopy types of smooth manifolds, elliptic manifolds would share this property.

As additional motivation we mention L\'arusson's holomorphic homotopy theory \cite{Larusson:2004,Larusson:2005}, in which ellipticity is a sufficient condition for a complex manifold to be fibrant. Given a Stein manifold, we may ask whether it possesses a fibrant model that is represented by a complex manifold, and this is almost exactly the question of whether there exists an acyclic embedding of the Stein manifold into an elliptic Stein manifold.

We now describe the content of the paper.

By a \emph{circular domain} we mean a domain given by removing a finite number of closed, pairwise disjoint discs from the open unit disc (see Definition~\ref{def_circularDomain}). Note that in this paper we do not permit our circular domains to have punctures. Our main result (Theorem~\ref{thm_strongOka}) is what we term a \emph{strong Oka property} for embeddings of circular domains into the elliptic Stein manifold $\CC\times\CC^*$:
\begin{ntheorem}
Let $X$ be a circular domain. Then every continuous map $X \to \CC\times\CC^*$ is homotopic to a proper holomorphic embedding $X \to \CC\times\CC^*$.
\end{ntheorem}

In order to prove the main theorem, we begin in Section~\ref{sec_Wold} by modifying techniques of Wold to obtain what we refer to as a \emph{Wold embedding theorem} for embedding certain Riemann surfaces into $\CC\times\CC^*$ (Theorem~\ref{thm_Wold}). This result relies upon the fact that $\CC\times\CC^*$ has the \emph{density property} (Definition~\ref{def_densityProperty}), and therefore that the Anders\'en-Lempert theorem holds regarding the approximation of certain maps by automorphisms of $\CC\times\CC^*$ (Theorem~\ref{thm_AndersenLempertHolConvex}). Fatou-Bieberbach domains in $\CC\times\CC^*$ (Definition~\ref{def_FatouBieberbachDomain}) also play an essential role in the proof of Theorem~\ref{thm_Wold}. 

In Section~\ref{sec_strongOka} we use the Wold embedding theorem for $\CC\times\CC^*$ to prove a variant of the embedding theorem of Forstneri\v c and Wold, for embeddings of Riemann surfaces into $\CC\times\CC^*$ (Theorem~\ref{thm_embedRiemannSurface}). As the target is no longer contractible, it makes sense to compare the homotopy classes of the given embedding of the bordered Riemann surface and the embedding of its interior produced by the theorem, and we show that these classes are in fact equal. Using this result we then prove the strong Oka property for embeddings of circular domains into $\CC\times\CC^*$ (Theorem~\ref{thm_strongOka}). As a corollary of the results in Section~\ref{sec_strongOka}, we obtain a complete solution to the problem of acyclically embedding annuli (possibly degenerate), that is, open Riemann surfaces with non-trivial abelian fundamental group, into the elliptic Stein manifold $\CC\times\CC^*$ (Corollary~\ref{cor_acyclicEmbedding}). We obtain the following result (Corollary~\ref{cor_acyclicEmbedding2}) after combining Corollary~\ref{cor_acyclicEmbedding} with the embedding of the open disc into $\CC^2$ given by Kasahara and Nishino:

\begin{ncorollary}
Every open Riemann surface with abelian fundamental group embeds acyclically into a 2-dimensional elliptic Stein manifold.
\end{ncorollary}

To the best of my knowledge a proof of the Anders\'en-Lempert theorem for Stein manifolds with the density property has not appeared in the literature, so for the benefit of the reader a detailed proof is provided in an appendix, following the proof for $\CC^n$ given by Forstneri\v c and Rosay \cite{Forstneric:1993,Forstneric:1994}.

I wish to thank Finnur L\'arusson for many helpful discussions during the preparation of this paper, and Frank Kutzschebauch for introducing to me the embedding techniques of Wold. Franc Forstneri\v c is also kindly thanked for providing access to a draft of his forthcoming book.

\section{A Wold embedding theorem for $\CC\times\CC^*$}
\label{sec_Wold}
\noindent Let $X$ be a non-compact Riemann surface that is the interior of a bordered Riemann surface $\overline{X}$ with a finite number of boundary components, all of which are non-compact. As mentioned in Section~\ref{sec_introduction}, by a \emph{bordered Riemann surface} we mean a two-dimensional smooth manifold with boundary, equipped with a complex structure. In this section we show that if $\overline{X}$ can be properly holomorphically embedded into $\CC\times\CC^*$ in such a way that the image of its boundary satisfies a certain \emph{nice projection property}, then $X$ itself can be properly holomorphically embedded into $\CC\times\CC^*$. Recall that all embeddings we discuss are both holomorphic and proper.

Let $\Delta_r \subset {\mathbb C}$ denote the open disc of radius $r > 0$ centred at the origin. For $r > 0$ let $A_r \subset {\mathbb C}^*$ denote the annulus $A_r = \{ w \in {\mathbb C}^* : 1/(r+1) < |w| < r + 1\}$. We call the open set $P_r = \Delta_r \times A_r \subset {\mathbb C} \times {\mathbb C}^*$ a \emph{cylinder} of radius $r$. By taking a sequence of cylinders whose radii form an increasing unbounded sequence we obtain an exhaustion of ${\mathbb C} \times {\mathbb C}^*$ by relatively compact open sets whose closures $\overline{P}_r$ are ${\mathscr O}({\mathbb C}\times{\mathbb C}^*)$-convex.

We use $\pi_1 : {\mathbb C}\times{\mathbb C}^* \to {\mathbb C}$ and $\pi_2 : {\mathbb C}\times{\mathbb C}^* \to {\mathbb C}^*$ to denote projection onto the first and second components respectively. The following definition is essentially the same as that used implicitly in \cite{Wold:2006,Wold:2006a} and stated explicitly in \cite{Kutzschebauch:2009}, except that we have added a condition relating to the injectivity of $\pi_1$ on a certain set. We believe this condition is required in the proofs of Lemma~2.2 in \cite{Kutzschebauch:2009} and Lemma~1 in \cite{Wold:2006a}.

\begin{definition}
\label{niceprojectionproperty}
Let $\gamma_1, \dots, \gamma_m$ be pairwise disjoint, smoothly embedded curves in ${\mathbb C}\times{\mathbb C}^*$, where each $\gamma_j$ maps either $[0,\infty)$ or $(-\infty,\infty)$ into ${\mathbb C}\times{\mathbb C}^*$. For each $j$ let $\Gamma_j\subset \CC\times\CC^*$ be the image of $\gamma_j$ and set $\Gamma = \bigcup\limits_{j=1}^{m} \Gamma_j$. We say that the collection $\gamma_1,\dots,\gamma_m$ has the \emph{nice projection property} if there is a holomorphic automorphism $\alpha \in \Aut({\mathbb C}\times{\mathbb C}^*)$ such that, if $\beta_j = \alpha\circ\gamma_j$ and $\Gamma' = \alpha(\Gamma)$, the following conditions hold:
\begin{enumerate}
\item[{\rm (1)}] \label{niceprojectionproperty1} $\lim\limits_{|t|\to\infty}|\pi_1(\beta_j(t))| = \infty$ for $j = 1,\dots,m$.
\item[{\rm (2)}] \label{niceprojectionproperty2} There exists $M \ge 0$ such that for all $r \geq M$:
\begin{enumerate}
\item[{\rm (a)}] ${\mathbb C}\setminus(\pi_1(\Gamma')\cup\overline\Delta_r)$ does not contain any relatively compact connected components.
\item[{\rm (b)}] $\pi_1$ is injective on $\Gamma' \setminus \pi_1^{-1}(\Delta_r)$.
\end{enumerate}
\end{enumerate}
\end{definition}

It is immediate from the definition that the nice projection property is invariant under automorphisms of $\CC\times\CC^*$. It is also clear that the nice projection property is independent of the parametrisation of $\gamma_1,\dots,\gamma_m$ so that we may refer to the set $\Gamma$ as having the nice projection property. Finally, we note that condition (\ref{niceprojectionproperty1}) implies that the restriction of $\pi_1$ to $\Gamma'$ is a proper map into ${\mathbb C}$.

Let $n > 1$. A \emph{Fatou-Bieberbach domain} is a proper open subset of $\CC^n$ that is biholomorphic to $\CC^n$. In the following definition we generalise this concept to domains in $\CC\times\CC^*$.

\begin{definition}
\label{def_FatouBieberbachDomain}
Let $\Omega$ be a proper open subset of $\CC\times\CC^*$. We say $\Omega$ is a \emph{Fatou-Bieberbach domain in $\CC\times\CC^*$} if there exists a biholomorphism $\phi : \Omega \to \CC\times\CC^*$.
\end{definition}

The main result of this section is the following theorem, proved by Wold for $\CC^2$ in \cite{Wold:2006,Wold:2006a}, and with interpolation by Kutzschebauch, L\o w and Wold in \cite{Kutzschebauch:2009}.

\begin{theorem}[Wold embedding theorem for $\CC\times\CC^*$]
\label{thm_Wold}
Let $X$ be an open Riemann surface and $K \subset X$ be a compact set. Suppose that $X$ is the interior of a bordered Riemann surface $\overline{X}$ whose boundary components are non-compact and finite in number.  If there is an embedding $\psi : \overline{X} \to \CC\times\CC^*$ such that $\psi(\partial\overline{X})$ has the nice projection property, then there exists an embedding $\sigma:X \to \CC\times\CC^*$ that approximates $\psi$ uniformly on $K$. In fact, there is a Fatou-Bieberbach domain $\Omega$ in $\CC\times\CC^*$ and a biholomorphism $\phi:\Omega \to \CC\times\CC^*$ such that $\psi(X) \subset \Omega$ and $\psi(\partial\overline{X}) \subset \partial \Omega$, and then we may take $\sigma = \phi \circ \psi|_X$.
\end{theorem}

In Section~\ref{sec_strongOka} we will show that the embedding $\sigma$ so constructed is homotopic to $\psi|_X$.

In proving Theorem~\ref{thm_Wold} we will follow the argument given in \cite{Kutzschebauch:2009}, making a number of modifications so that the proof holds with target space $\CC\times\CC^*$ for the embedding, rather than $\CC^2$. The proof will require a number of preliminary results and definitions, and we begin by reminding the reader what it means for a vector field to be \emph{$\RR$-complete}.

\begin{definition}
\label{def_completeVectorField}
Let $V$ be a smooth vector field on a smooth manifold $X$ with flow $\phi_t$. The \emph{maximal domain} of $V$ is the largest subset of $\RR\times X$ on which $\phi_t$ is defined. If the maximal domain of $V$ equals all of $\RR\times X$ then $V$ is said to be \emph{$\RR$-complete}.
\end{definition}

The following notion was first introduced by Varolin \cite{Varolin:2000,Varolin:2001} and further studied by T\' oth and Varolin \cite{Toth:2000,Toth:2006} and Kaliman and Kutzschebauch \cite{Kaliman:2008,Kaliman:2008a}. It generalises a property of $\CC^n$ of vital importance in the approximation of certain maps by automorphisms (see Theorem~\ref{thm_AndersenLempertHolConvex}).

\begin{definition}
\label{def_densityProperty}
Let $X$ be a complex manifold. We say $X$ has the \emph{density property} if the Lie algebra generated by the $\RR$-complete holomorphic vector fields on $X$ is dense in the Lie algebra of all holomorphic vector fields on $X$ in the compact-open topology.
\end{definition}

The results in this paper rely on the essential fact that $\CC\times\CC^*$ possesses the density property, as proved by Varolin \cite{Varolin:2001}.

\begin{definition}
\label{def_isotopy}
Let $X$ be a complex manifold and $\Omega \subset X$ be an open set. Let $k \in \NN \cup \{\infty\}$. A \emph{$\Cont^k$-isotopy of injective holomorphic maps from $\Omega$ to $X$} is a $\Cont^k$ map $\Phi : [0,1] \times \Omega \to X$ such that for each fixed $t \in [0,1]$, the map $\Phi(t,\cdot) : \Omega \to X$ is an injective holomorphic map. We often write $\Phi_t$ for $\Phi(t, \cdot)$ and write $\Phi_t : \Omega \to X$ for the isotopy.
\end{definition}

Recall that a compact set $K$ in a complex manifold $X$ is said to be $\mathscr{O}(X)$-convex if $K = \widehat K_{\OO(X)}$, where
\[
	\widehat K_{\OO(X)} = \{ x \in X : \lvert f(x) \rvert \le \sup\limits_{y\in K}\lvert f(y) \rvert \text{ for all } f \in \OO(X)\}
\]
is the holomorphically convex hull of $K$ in $X$. In particular, when $X = \CC^n$, $\widehat K_{\OO(\CC^n)}$ equals the polynomially convex hull of $K$.

The following result on the approximation of injective holomorphic maps by automorphisms was proved by Forstneri\v c and Rosay for $X = \CC^n, n>1$, in \cite{Forstneric:1993,Forstneric:1994}. The result follows from their generalisation (also in \cite{Forstneric:1993,Forstneric:1994}) of the Anders\'en-Lempert theorem, first proved in a different form by Anders\'en and Lempert \cite{Andersen:1992}. These two results from \cite{Forstneric:1993,Forstneric:1994}, together with stronger results by Forstneri\v c and L\o w \cite{Forstneric:1997}, and Forstneri\v c, L\o w and \O vrelid \cite{Forstneric:2001}, are collectively known as \emph{Anders\'en-Lempert theorems} for $\CC^n$. 

\begin{theorem}
\label{thm_AndersenLempertHolConvex}
Let $X$ be a Stein manifold with the density property. Let $\Omega \subset X$ be an open set and $\Phi_t : \Omega \to X$ a $\Cont^1$-isotopy of injective holomorphic maps such that $\Phi_0$ is the inclusion of $\Omega$ into $X$. Suppose $K \subset \Omega$ is a compact set such that $K_t = \Phi_t(K)$ is $\mathscr{O}(X)$-convex for every $t \in [0,1]$. Then $\Phi_1$ can be uniformly approximated on $K$ by holomorphic automorphisms of $X$ with respect to any Riemannian distance function on $X$.
\end{theorem}

As discussed in Section~\ref{sec_introduction}, the proof of Theorem~\ref{thm_AndersenLempertHolConvex} follows closely the argument given for $\CC^n$ in \cite{Forstneric:1993,Forstneric:1994}. To the best of my knowledge a proof has not appeared in the literature, so, for the benefit of the reader, a detailed argument is provided in an appendix.

The following lemma is easy in the case of an embedded real-analytic curve, yet becomes surprisingly subtle if the curve is only assumed to be smooth. With only minor modifications, the following argument gives the corresponding result for embedded smooth curves in $\CC^n, n > 1$.

\begin{lemma} 
\label{lem_convexNbhdBasis}
Let $\gamma : [0,1] \to \CC^2$ be a smooth embedding. Then the image of $\gamma$ has a neighbourhood basis of open sets each biholomorphic to a convex set in $\CC^2$.
\end{lemma}
\begin{proof}
We outline an argument extracted from Rosay's proof of the main result in his paper \cite{Rosay:1993}, referring the reader there for complete details.

Begin by fixing a neighbourhood $V$ of $\gamma([0,1])$. We will construct a neighbourhood $V'\subset V$ of $\gamma([0,1])$ and a biholomorphism $\chi : W \to V'$, where $W$ is a convex set in $\CC^2$. Identify $[0,1]$ with $J = [0,1] \times \{0\} \subset \CC^2$ and extend $\gamma$ to a map $\tilde{\gamma}$ defined on $\CC^2$ that agrees with $\gamma$ on $J$ by means of the formula
\[
	\tilde{\gamma}(z,w) = \gamma(z) + w\alpha(z)\,.
\]
Here $\gamma$ is now a compactly supported almost-analytic extension of $\gamma$ to $\CC$ (meaning that $\bar\partial\gamma$ vanishes to infinite order along $\RR$, where $\bar\partial = \partial/\partial \bar z$) and $\alpha$ is a holomorphic map into $\CC^2$ such that for all $t \in [0,1]$ the vectors $\dot\gamma(t)$ and $\alpha(t)$ are linearly independent. Then $\tilde\gamma$ defines a diffeomorphism from a neighbourhood of $J$ into $\CC^2$, $\tilde\gamma$ is holomorphic in $w$ and $\bar\partial\tilde\gamma$ vanishes to infinite order along $\RR\times\CC$.

Let $\lambda$ be a smooth function on $\RR$ such that $\lambda(x) = 1$ for $x\le 2$ and $\lambda(x) = 0$ for $x \ge 3$. For $j \in \NN$ define $\lambda_j : \CC\to \RR$ by
\[
	\lambda_j(x+iy) = \lambda(jy)\lambda(-jy)\,,\quad x + iy \in \CC\,.
\]
Then $\lambda_j$ has support in the horizontal strip $|y| \le 3/j$, equals $1$ if $|y| \le 2/j$, and satisfies the following estimates on its derivatives: for each $k \in \NN$ there exists $C > 0$ such that
\[
\left\lvert \frac{\partial^k \lambda_j}{\partial y^k}\right\rvert \le C j^k\,.
\]

Rosay shows that we may find smooth maps $u_j : \CC^2 \to \CC^2$, $j \in \NN$, that solve the differential equation $\bar\partial u_j = \lambda_j(z)\bar\partial\tilde\gamma$, such that $u_j$ tends to $0$ in the $\Cont^\infty$ topology as $j \to \infty$. The convergence is sufficiently rapid that for every $k, l \in \NN$ and every compact set $H \subset \CC^2$, there exists $C > 0$ and $j_0 \in \NN$ such that for $j \ge j_0$,
\begin{equation}
\label{eqn_uGoToZero}
    \lVert u_j \rVert_{\Cont^k(H)} \le \frac{C}{j^{l}}\,.
\end{equation}

Fix a bounded neighbourhood $U$ of $J$ in $\CC^2$ sufficiently small so that $\tilde\gamma$ is a diffeomorphism of $U$ into $\CC^2$. For $j$ large enough, $\tilde\gamma - u_j$ will also be a diffeomorphism from $U$ into $\CC^2$. By taking $j$ sufficiently large we can therefore ensure $\tilde\gamma - u_j$ is a diffeomorphism on the convex set $\Omega_j\subset U$, where
\[
    \Omega_j = \{ (z,w) \in \CC^2 : \dist((z,w),J) < \frac{1}{j}\}
\]
is the open $1/j$-neighbourhood of $J$ in $\CC^2$. If necessary, we may take $j$ larger to ensure that $(\tilde\gamma - u_j)(\Omega_j) \subset V$. By the definition of $\Omega_j$ we have that $\lambda_j(z) = 1$ if $(z,w) \in \Omega_j$, so that $\tilde\gamma - u_j$ is also holomorphic on $\Omega_j$ and hence a biholomorphism onto its image. From (\ref{eqn_uGoToZero}) we have
\[
	\lVert (\tilde\gamma - u_j) - \tilde\gamma \rVert_{\Cont^1(\Omega_j)} \le \frac{C}{j^2}
\]
for some $C > 0$ and $j$ sufficiently large. We may therefore apply Lemma~II.2.1 from \cite{Rosay:1993} to conclude that for all sufficiently large $j$, $\Gamma \subset (\tilde\gamma - u_j)(\Omega_j)$.

Fix some sufficiently large $j$ so that all the conditions in the previous paragraph hold and let $\chi = \tilde\gamma - u_j$. Let $W = \Omega_j$ and $V' = \chi(\Omega_j)$. Then $\chi : W \to V'$ is a biholomorphism such that $\Gamma \subset V'\subset V$, and $W$ is convex.
\end{proof}

The following corollary will be required in the proof of Lemma~\ref{existauto}, where we construct an isotopy of injective holomorphic maps that send a finite number of compact curve segments outside a large cylinder.

\begin{corollary}
\label{cor_contractCurve}
Let $\gamma : [0,1] \to \CC^2$ be a smooth embedding and $V \subset \CC^2$ be an open neighbourhood of $\Gamma = \gamma([0,1])$. Fix a point $p \in \Gamma$. Then there exists an open neighbourhood $V' \subset V$ of $\Gamma$ such that for all $\epsilon > 0$ there exists a $\Cont^1$-isotopy of injective holomorphic maps $\Phi_t : V' \to V'$ satisfying:
\begin{enumerate}
\item[{\rm (1)}] $\Phi_0 = \id$.
\item[{\rm (2)}] $\Phi_t(p) = p$ for all $t \in [0,1]$.
\item[{\rm (3)}] $\Phi_1(V') \subset B(p,\epsilon)$,
\end{enumerate}
where $B(p,\epsilon)$ is the open ball of radius $\epsilon$ centred at $p \in \CC^2$.
\end{corollary}
\begin{proof}
Using Lemma~\ref{lem_convexNbhdBasis} we obtain a open neighbourhood $V' \subset V$ of $\Gamma$ and a biholomorphism $\chi : W \to V'$, where $W \subset \CC^2$ is a relatively compact convex open set. For $0 < \delta < 1$ define an isotopy $\widetilde\Phi_t : W \to W$ by
\[
	\widetilde\Phi_t(\zeta) = \chi^{-1}(p) + (\zeta - \chi^{-1}(p))(1 - t(1-\delta))\,, \quad \zeta \in W\,.
\]
Then $\widetilde\Phi_0$ is the identity, and $\widetilde\Phi_t$ linearly contracts $W$ for $t > 0$, always keeping $\chi^{-1}(p) \in W$ fixed. Choose $\delta > 0$ sufficiently small so that $\widetilde\Phi_1(W) \subset \chi^{-1}(B(p, \epsilon))$.

Let $\Phi_t = \chi \circ \widetilde\Phi_t \circ \chi^{-1}$. Then $\Phi_t : V' \to V'$ is a $\Cont^1$-isotopy of injective holomorphic maps that satisfies conditions (1)--(3).
\end{proof}

The following well-known lemma establishes a correspondence between the notions of $\OO(M)$-convexity of a subset $K$ of a complex manifold $M$ and the $\OO(S)$-convexity of the image of $K$ under an embedding $\varphi : M \to S$, where $S$ is a Stein manifold.

\begin{lemma}
\label{convexity}
Let $\varphi:M \to S$ be an embedding of the complex manifold $M$ into the Stein manifold $S$, and let $K\subset M$. Then
\begin{equation*}
\varphi(\widehat{K}_{{\mathscr O}(M)}) = \widehat{\varphi(K)}_{{\mathscr O}(S)}\,.
\end{equation*}
\end{lemma}
\begin{proof}
Let $p \in \widehat{K}_{{\mathscr O}(M)}$ and $f\in{\mathscr O}(S)$. Then $f\circ\varphi \in {\mathscr O}(M)$ so that $\lvert f(\varphi(p))\rvert = \lvert(f\circ\varphi)(p)\rvert \leq \lVert f\circ\varphi\rVert_K = \lVert f\rVert_{\varphi(K)}$, and therefore $\varphi(p) \in \widehat{\varphi(K)}_{{\mathscr O}(S)}$.

Conversely, let $q\in\widehat{\varphi(K)}_{{\mathscr O}(S)}$. If $q\notin\varphi(M)$ then there exists $f \in {\mathscr O}(S)$ such that $f|_{\varphi(M)}\equiv 0$ and $f(q) = 1$, so that $q\notin \widehat{\varphi(K)}_{{\mathscr O}(S)}$, a contradiction. Thus $q\in \varphi(M)$. Now suppose $q = \varphi(p)$, $p \in M$, and let $g \in {\mathscr O}(M)$. The holomorphic function $g\circ\varphi^{-1}$ defined on the closed submanifold $\varphi(M)\subset S$ then extends to an entire function $\tilde{g}\in{\mathscr O}(S)$. It follows that $\lvert g(p)\rvert = \lvert\tilde{g}(\varphi(p))\rvert = \lvert\tilde{g}(q)\rvert \leq \lVert \tilde{g} \rVert_{\varphi(K)} = \lVert g\rVert_K$ because $q\in\widehat{\varphi(K)}_{{\mathscr O}(S)}$. Hence $p \in \widehat{K}_{{\mathscr O}(M)}$ and $q = \varphi(p) \in \varphi(\widehat{K}_{{\mathscr O}(M)})$.
\end{proof}

Given a connected Stein manifold $M$, by the Bishop-Narasimhan-Remmert embedding theorem there exists an embedding $\varphi : M \to \CC^n$ for some $n \in \NN$. We may therefore apply the preceding result to show that existing lemmas involving polynomial convexity of certain sets (that is, $\OO(\CC^n)$-convexity) continue to hold true with respect to the $\OO(M)$-convexity of sets in a connected Stein manifold $M$. We present the first of two such results below, a generalisation of a theorem of Stolzenberg \cite{Stolzenberg:1966}.

\begin{lemma}
\label{Stolzenberg}
Let $\Gamma_1,\dots,\Gamma_m$ be compact, smooth, pairwise disjoint, embedded curves in a connected Stein manifold $M$. Let $K \subset M$ be an $\OO(M)$-convex compact set, disjoint from $\Gamma = \bigcup\limits_{j=1}^m \Gamma_j$. Then the set $K \cup \Gamma$ is $\OO(M)$-convex.
\end{lemma}
\begin{proof}
Let $\varphi : M \hookrightarrow \CC^n$ be an embedding. Then $\varphi(\Gamma_1),\dots,\varphi(\Gamma_m)$ is a collection of compact, smooth, pairwise disjoint, embedded curves in $\CC^n$. By Lemma~\ref{convexity}, $\widehat{\varphi(K)}_{\OO(\CC^n)} = \varphi(\widehat{K}_{\OO(M)}) = \varphi(K)$ so that $\varphi(K)$ is a polynomially convex compact set, and $\varphi(K) \cap \varphi(\Gamma) = \varnothing$. By Stolzenberg's theorem \cite{Stolzenberg:1966}, $\varphi(K) \cup \varphi(\Gamma)$ is polynomially convex and hence
\[
  \varphi(K \cup \Gamma)^{\widehat{\quad}}_{\OO(\CC^n)} = (\varphi(K) \cup \varphi(\Gamma))^{\widehat{\quad}}_{\OO(\CC^n)} = \varphi(K) \cup \varphi(\Gamma) = \varphi(K \cup \Gamma)\,.
\]
Again applying Lemma~\ref{convexity} we have
\[
	\widehat{K \cup \Gamma}_{\OO(M)} = \varphi^{-1}(\varphi(K \cup \Gamma)^{\widehat{\quad}}_{\OO(\CC^n)}) = \varphi^{-1}(\varphi(K \cup \Gamma)) = K \cup \Gamma
\]
as required.
\end{proof}

Using Theorem~\ref{thm_AndersenLempertHolConvex}, Corollary~\ref{cor_contractCurve} and Lemma~\ref{Stolzenberg} we obtain the following main technical result required in the proof of Theorem~\ref{thm_Wold}. Our proof follows that of Wold in \cite[Lemma~1]{Wold:2006a}, adapted to work in $\CC\times\CC^*$.

\begin{lemma}
\label{existauto}
Equip $\CC\times\CC^*$ with a Riemannian distance function $d$. Let $K \subset {\mathbb C}\times{\mathbb C}^*$ be an ${\mathscr O}({\mathbb C}\times{\mathbb C}^*)$-convex compact set and let $\gamma_1,\dots,\gamma_m$ be pairwise disjoint, smoothly embedded curves in ${\mathbb C}\times{\mathbb C}^*$ satisfying the nice projection property. Let $\Gamma_j$ be the image of $\gamma_j$ in $\CC\times\CC^*$, $j = 1,\dots,m$, and set $\Gamma = \bigcup\limits_{j=1}^m \Gamma_j$. Suppose that $\Gamma \cap K = \varnothing$. Then, given $r > 0$ and $\epsilon > 0$, there exists $\psi \in \Aut({\mathbb C}\times{\mathbb C}^*)$ such that the following conditions are satisfied:
\begin{itemize}
\item[(a)] $\sup\limits_{\zeta\in K}d(\psi(\zeta),\zeta) < \epsilon$.
\item[(b)] $\psi(\Gamma) \subset {\mathbb C}\times{\mathbb C}^* \setminus\overline P_r$.
\item[(c)] $\psi$ is homotopic to the identity map.
\end{itemize}
\end{lemma}

\begin{proof}
We first show that we may assume that the automorphism $\alpha$ in Definition~\ref{niceprojectionproperty} has already been applied, so that the conditions in the nice projection property hold directly for each $\gamma_j$ and $\Gamma$. Indeed, given $K$, $\Gamma$, $r$ and $\epsilon$ as above, choose a slightly larger ${\mathscr O}({\mathbb C}\times{\mathbb C}^*)$-convex compact set $\widetilde{K}$ that contains $K$ in its interior and such that we still have $\widetilde{K} \cap \Gamma = \varnothing$ (take a small neighbourhood $\omega$ of $K$ disjoint from $\Gamma$ and take the closure of the component of an analytic polyhedron containing $K$ and contained within $\omega)$. Choose $r'$ sufficiently large so that $\alpha(P_r) \subset P_{r'}$. By the uniform continuity of $\alpha^{-1}$ on the compact set $\alpha(\widetilde{K})$ there is $\delta > 0$ such that for $x, y \in \alpha(\widetilde{K}), d(x,y) < \delta$ implies $d(\alpha^{-1}(x),\alpha^{-1}(y)) < \epsilon$. Set $\epsilon' = \min\{d(\alpha(K), {\mathbb C}\times{\mathbb C}^*\setminus \alpha(\widetilde{K})), \delta\}$. The lemma then gives $\widetilde{\psi}\in \Aut({\mathbb C}\times{\mathbb C}^*)$ such that $\sup\limits_{\zeta\in\alpha(\widetilde{K})}d(\widetilde{\psi}(\zeta),\zeta) < \epsilon'$ and $\widetilde{\psi}(\alpha(\Gamma))\subset {\mathbb C}\times{\mathbb C}^*\setminus \overline P_{r'}$. Letting $\psi = \alpha^{-1}\circ\widetilde{\psi}\circ\alpha$, we have
\begin{equation*}
\psi(\Gamma) = \alpha^{-1}(\widetilde{\psi}(\alpha(\Gamma))) \subset \alpha^{-1}({\mathbb C}\times{\mathbb C}^* \setminus \overline P_{r'}) = {\mathbb C}\times{\mathbb C}^*\setminus \alpha^{-1}(\overline P_{r'}) \subset {\mathbb C}\times{\mathbb C}^*\setminus \overline P_r\, .
\end{equation*}
Finally, let $\zeta\in K$, so that $\alpha(\zeta) \in \alpha(K)$. Then $d(\widetilde{\psi}(\alpha(\zeta)), \alpha(\zeta)) < \epsilon'$ so that $\widetilde{\psi}(\alpha(\zeta)) \in \alpha(\widetilde{K})$ and hence $d(\psi(\zeta), \zeta) < \epsilon$.

Next we show how to guarantee condition (c). Suppose that the projection $\pi_1(K)$ of $K$ onto $\CC$ is contained in a closed disk $D \subset \CC$. Pick $z_0 \in \CC \setminus (D \cup \pi_1(\Gamma))$ and let $S = \{(z_0,e^{2\pi it}):t \in [0,1]\}\subset \CC\times\CC^*$ be a loop about the missing line in $\CC\times\CC^*$. The set $K \cup S$ is then still $\OO(\CC\times\CC^*)$-convex and compact. Indeed, if $p \in \CC\times\CC^*$ and $\pi_1(p) \neq z_0$, take $f \in \OO(\CC\times\CC^*)$ such that $\lvert f(p)\rvert > 1$ and $\lVert f \rVert_K < 1$. As $D \cup \{z_0, \pi_1(p)\}$ is Runge in $\CC$, find $g \in \OO(\CC)$ that approximates 1 on $D \cup \{\pi_1(p)\}$ and 0 at $z_0$, so that $h(z,w) = f(z,w)g(z) \in \OO(\CC\times\CC^*)$ satisfies $\lvert h(p) \vert > 1$ and $\lVert h \rVert _{K\cup S} < 1$. On the other hand, if $\pi_1(p) = z_0$, then take $f \in \OO(\CC^*)$ such that $\lVert f \rVert _{\pi_2(S)} < 1$ and $\lvert f(\pi_2(p))\rvert > 1$, and take $g \in \OO(\CC)$ that approximates 0 on $D$ and 1 at $z_0$. Then $h(z,w) = f(w)g(z) \in \OO(\CC\times\CC^*)$ satisfies $\lvert h(p) \vert > 1$ and $\lVert h \rVert_{K \cup S} < 1$.

Now assume $\psi$ has been produced satisfying conditions (a) and (b), where $K$ has been replaced by $K \cup S$. There are only two possible homotopy classes of automorphisms of $\CC\times\CC^*$, determined by whether the orientation of a loop about the missing line is preserved or reversed. By requiring $\epsilon < 1$, a convex linear combination will interpolate between $\psi(S)$ and $S$ without passing through the missing line in $\CC\times\CC^*$, showing that $\psi$ is homotopic to $\id \in \Aut(\CC\times\CC^*)$.

We will now construct the desired automorphism $\psi$ in two stages, as the composition of two automorphisms of $\CC\times\CC^*$. For the first stage, as before, we take a slightly larger ${\mathscr O}(\CC\times\CC^*)$-convex compact set $K'$ that contains $K$ in its interior and such that we still have $K' \cap \Gamma = \varnothing$. Shrink $\epsilon$ if necessary to ensure that $\epsilon/2 < d(K, \CC\times\CC^* \setminus K')$. We may assume that $r \geq M$, where $M$ is determined by the nice projection property for $\gamma_1,\dots,\gamma_m$. We may also assume that $K' \subset \Delta_r \times {\mathbb C}^*$ and that $\gamma_j(0) \in \Delta_r \times \CC^*$ for $j = 1,\dots,m$. We set $\widetilde{\Gamma} = \Gamma \cap (\overline{\Delta}_r \times {\mathbb C}^*) = (\pi_1 |_\Gamma)^{-1}(\overline{\Delta}_r)$, which is compact due to the properness of $\pi_1|_\Gamma$. In fact, by the nice projection property, $\widetilde\Gamma$ will have precisely $m$ connected components $\widetilde\Gamma_1,\dots, \widetilde\Gamma_m$, where each $\widetilde\Gamma_j = \Gamma_j \cap (\overline{\Delta}_r \times {\mathbb C}^*)$ is a compact smoothly embedded curve. Choose an open neighbourhood $U_0$ of $K'$ and open neighbourhoods $U_1, \dots, U_m$ of each of the components of $\widetilde{\Gamma}$ so that the sets $U_0, \dots, U_m$ are all pairwise disjoint. For each $j= 1, \dots, m$, pick a point $p_j \in \widetilde\Gamma_j$. By Corollary~\ref{cor_contractCurve} we may assume that each $U_j$ can be contracted within itself by a $\Cont^1$-isotopy of injective holomorphic maps leaving $p_j$ fixed, so that the final image of $U_j$ lies within an arbitrarily small neighbourhood of $p_j$.

We now take the $\Cont^0$-isotopy of injective holomorphic maps defined on $\bigcup\limits_{j=0}^m U_j$ that is the inclusion on $U_0$ for $t \in [0,1]$, and that uses the isotopy on each $U_j$, $j=1,\dots,m$ to shrink $U_j$ for $t \in [0,1/2]$ and then translates each shrunken image of $U_j$ along the curve $\Gamma_j$ for $t \in [1/2, 1]$ so that at $t=1$ the images of each of the $U_j$ lie entirely outside of $\overline{\Delta}_r \times \CC^*$ (the isotopies of the $U_j$ can be chosen with sufficiently small final images to ensure that the translated shrunken images of the $U_j$ are always pairwise disjoint and do not meet $U_0$ during this process). The isotopy is $\Cont^1$ except at $t=1/2$, and can be made $\Cont^1$ everywhere by reparametrising the unit interval, giving a $\Cont^1$-isotopy of injective holomorphic maps on $\bigcup\limits_{j=0}^m U_j$ that leaves $K'$ fixed for all $t \in [0,1]$, and moves each component of $\widetilde{\Gamma}$ outside of $\overline{\Delta}_r \times \CC^*$ at $t = 1$. By Lemma~\ref{Stolzenberg} the image of $K'\cup\widetilde{\Gamma}$ under the isotopy is $\OO(\CC\times\CC^*)$-convex for each $t \in [0,1]$. We may therefore apply Theorem~\ref{thm_AndersenLempertHolConvex} to give $\alpha \in \Aut(\CC\times\CC^*)$ such that the following conditions hold:
\begin{itemize}
\item[(a)] $\sup\limits_{\zeta\in K'} d(\alpha(\zeta), \zeta) < \epsilon / 2$.
\item[(b)] $\alpha(\widetilde{\Gamma}) \subset \CC\times\CC^* \setminus \overline{P}_r$.
\end{itemize}

The automorphism $\alpha$ moves all of $\widetilde{\Gamma}$ outside of $\overline{P}_r$, but may also move points from $\Gamma \setminus \widetilde{\Gamma}$ into $\overline P_r$ that were not there previously. Let $\Gamma_r = \{ \zeta \in \Gamma : \alpha(\zeta) \in \overline{P}_r\} = \Gamma \cap \alpha^{-1}(\overline{P}_r)$. By construction, $\pi_1(\Gamma_r) \subset \pi_1(\Gamma)\setminus \overline{\Delta}_r$, and recall that $r$ was chosen so that $\CC\setminus (\pi_1(\Gamma) \cup \overline{\Delta}_r)$ has no bounded connected components and such that $\pi_1$ is injective on $\Gamma$ outside of $\Delta_r \times \CC^*$. We will now construct $\beta \in \Aut(\CC\times\CC^*)$ that is approximately the identity on $\overline \Delta_r \times \CC^*$, where $\Gamma$ already avoids $\alpha^{-1}(\overline{P}_r)$, and that moves the set $\Gamma \setminus \widetilde\Gamma$ so as to avoid $\alpha^{-1}(\overline{P}_r)$. We look for an automorphism of the form $\beta(z,w) = (z,w e^{g(z)})$, where $(z,w) \in \CC\times\CC^*$ and $g \in \mathscr{O}(\CC)$. In order to obtain $\beta$ we construct a continuous map $\tilde\beta(z,w) = (z, w f(z))$ from $(\overline\Delta_r \cup \pi_1(\Gamma))\times\CC^*$ to $\CC\times\CC^*$, where $f : \overline\Delta_r \cup \pi_1(\Gamma) \to \CC^*$ is continuous on $\overline\Delta_r \cup\pi_1(\Gamma)$ and holomorphic on $\Delta_r$, and then choose $g \in \OO(\CC)$ so that $e^g$ approximates $f$ uniformly on an appropriate set.

Note that $\alpha^{-1}(\overline{P}_r)$ is compact so that we may find $s > r$ such that $\alpha^{-1}(\overline{P}_r) \subset \overline{P}_s$. If $(z,w) \in \CC\times\CC^*$ such that $|z| > s$, then since $\pi_1(\beta(z,w))  = z$, we have $\beta(z,w) \notin \alpha^{-1}(\overline{P}_r)$. The same statement is true for $\tilde\beta$. By setting $f(z) = 1$ for $z \in \overline\Delta_r$, we have that $\tilde\beta = \id$ on $\overline\Delta_r \times\CC^*$. It therefore remains to define $f$ on $\pi_1(\Gamma)\setminus\overline\Delta_r$ to ensure that $\tilde\beta$ moves $\Gamma' = \Gamma \cap (\overline{P}_s \setminus (\Delta_r \times\CC^*))$ so as to avoid $\alpha^{-1}(\overline{P}_r)$.

As we are restricting our attention to the set $\Gamma'$ we may now assume that each $\gamma_j$ has domain $[0,\infty)$ by taking each $\gamma_j$ that has domain $(-\infty,\infty)$ and splitting it into two curves $\gamma_j(t + \delta)$ and $\gamma_j(-t -\delta)$, $t \in [0,\infty)$, where $\delta > 0$ is sufficiently small so that $\gamma_j(\delta),\gamma_j(-\delta) \in \overline{\Delta}_r\times\CC^*$. Let $k\ge m$ be the new total number of curves $\gamma_j$.

For each $\gamma_j(t) = (z_j(t),w_j(t))$, choose $t^j_0 > 0$ such that $\pi_1(\gamma_j(t_0^j)) = z_j(t^j_0) \in \partial \Delta_r$, $z_j(t) \in \overline{\Delta}_r$ for $t < t_0^j$ and $z_j(t) \in \CC\setminus\overline{\Delta}_r$ for $t > t_0^j$. This is possible because $\pi_1$ is injective on $\Gamma\setminus\pi_1^{-1}(\Delta_r)$ and $\CC\setminus(\pi_1(\Gamma)\cup\overline{\Delta}_r)$ has no relatively compact connected components. Note that $\gamma_j(t^j_0) \notin \alpha^{-1}(\overline{P}_r)$ since $\pi_1(\gamma_j(t^j_0)) = z_j(t^j_0) \in \partial\Delta_r \subset \overline{\Delta}_r$, hence $\gamma_j(t^j_0)\in\widetilde{\Gamma}$, and $\widetilde{\Gamma}\cap\alpha^{-1}(\overline{P}_r) = \varnothing$.

Let $B_j = \sup\{|w_j(t)| : t \ge t^j_0 \text{ such that } z_j(t) \in \overline{\Delta}_s\}$ and $b_j = \inf\{|w_j(t)| : t \ge t^j_0 \text{ such that } z_j(t) \in \overline{\Delta}_s\}$, $j = 1,\dots,k$. As $\Gamma'$ is compact, each $B_j$ is finite and each $b_j$ is positive. Let $B = \max\{B_j\}$ and $b = \min\{b_j\}$.

Set $L_j = \{z_j(t^j_0)\}\times\CC^*$ and $K_j = L_j \cap \alpha^{-1}(\overline{P}_r)$. Note that all the $L_j$, and hence also all the $K_j$, are distinct by injectivity of $\pi_1$ on $\Gamma \setminus \Delta_r\times\CC^*$.

Since $\alpha^{-1}(\overline{P}_r)$ is $\OO(\CC\times\CC^*)$-convex, each $L_j \setminus K_j$ has no relatively compact connected components. In particular, each $L_j \setminus K_j$ is either connected, or has exactly two connected components, one bounded and the other unbounded. For each $j$, the point $(z_j(t^j_0),w_j(t^j_0))$ lies in one of the connected components of $L_j \setminus K_j$. We may therefore find continuous paths $c_j : [0,1] \to L_j \setminus K_j$ such that $c_j(0) = \gamma_j(t^j_0) = (z_j(t^j_0), w_j(t^j_0))$ and such that $c_j(1) \in L_j \setminus (\{z_j(t^j_0)\} \times \overline{A}_T)$, where $T > 0$ is determined by $T + 1 = (s+1) \max\{B/|w_j(t^j_0)|, |w_j(t^j_0)|/b\}$ (recall the definition of the annulus $A_T$).

Now define $\tilde{c}_j : [0,1] \to \CC^*$ by $\tilde{c}_j(t) = \pi_2(c_j(t))/w_j(t^j_0)$, so that we have $c_j(t) = (z_j(t^j_0),w_j(t^j_0)\tilde{c}_j(t))$. Then $\tilde{c}_j(0) = 1$ and $w_j(t^j_0)\tilde{c}_j(1) \notin \overline{A}_T$. There is a neighbourhood $V_j$ of $c_j([0,1])$ in $\CC\times\CC^*$ that is contained in $\CC\times\CC^* \setminus \alpha^{-1}(\overline{P}_r)$. Then for sufficiently small $\delta > 0$, the curve $(z_j(t^j_0 + \delta t), w_j(t^j_0 + \delta t)\tilde{c}_j(t))$, $t \in [0,1]$, still lies in $V_j$ and we still have $w_j(t^j_0 + \delta) \tilde{c}_j(1) \notin \overline{A}_T$.

Define $f : \overline{\Delta}_r \cup \pi_1(\Gamma) \to \CC^*$ by
\begin{equation*}
f = \left\{
	\begin{array}{ll}
		1 & \text{on } \overline{\Delta}_r,\\
		\tilde{c}_j(t/\delta) & \text{at } z_j(t^j_0 + t) \text{ for } t \in [0,\delta], j = 1,\dots,k,\\
		\tilde{c}_j(1) & \text{at } z_j(t^j_0 + t) \text{ for } t > \delta, j = 1,\dots,k.
	\end{array} \right.
\end{equation*}
With this choice of $f$, as $t$ increases from $0$ to $\delta$, the point $\tilde\beta(\gamma_j(t^j_0 + t))$ avoids the set $\alpha^{-1}(\overline{P}_r)$ and at $t = \delta$ its second component lies outside of $\overline{A}_T$. However, we still need to ensure that following this, for $t \ge \delta$ such that $z_j(t^j_0 + t) \in \overline{\Delta}_s$, the second component $w_j(t^j_0 + t)\tilde{c}_j(1)$ of $\tilde\beta(\gamma_j(t^j_0 + t))$ does not re-enter into $\overline A_s$. This will in turn ensure that $\tilde\beta(\gamma_j(t^j_0 + t)) \notin \alpha^{-1}(\overline{P}_r) \subset \overline{P}_s$ for the same values of $t$.

The choice of $T$ made earlier ensures that if $t \ge t^j_0$ such that $z_j(t) \in \overline{\Delta}_s$, then either
\[
	\lvert w_j(t)\tilde{c}_j(1)\rvert \le B \lvert \tilde{c}_j(1)\rvert < B/(|w_j(t^j_0)|(T+1)) \le 1/(s+1)\,,
\]
if $\lvert w_j(t^j_0)\tilde{c}_j(1)\rvert < 1/(T+1)$, or
\[
	\lvert w_j(t)\tilde{c}_j(1)\rvert \ge b \lvert \tilde{c}_j(1)\rvert > b(T+1)/|w_j(t^j_0)| \ge s+1\,,
\]
if $\lvert w_j(t^j_0)\tilde{c}_j(1)\rvert > T+1$. That is, $w_j(t)\tilde{c}_j(1) \notin \overline{A}_s$ for all such $t$, as required.

The map $f$ is continuous on $\overline{\Delta}_r \cup \pi_1(\Gamma)$ and holomorphic on $\Delta_r$. Since $\overline{\Delta}_r\cup\pi_1(\Gamma)$ is simply connected, there is $\tilde{f} : \overline{\Delta}_r\cup\pi_1(\Gamma) \to \CC$ such that $e^{\tilde{f}} = f$, and $\tilde{f}$ is also continuous on $\overline{\Delta}_r\cup\pi_1(\Gamma)$ and holomorphic on $\Delta_r$. Given $\rho > 0$ we apply Mergelyan's theorem to approximate $\tilde{f}$ uniformly on the compact set $(\overline{\Delta}_r\cup\pi_1(\Gamma))\cap\overline{\Delta}_s$ by $g \in \OO(\CC)$ sufficiently accurately to ensure that, by uniform continuity of the exponential map on compact sets, we have $\lVert f - e^{g} \rVert_{(\overline{\Delta}_r\cup\pi_1(\Gamma))\cap\overline{\Delta}_s} < \rho$.

By choosing $\rho > 0$ sufficiently small we ensure that $\beta$ is sufficiently close to the identity on $\overline{\Delta}_r \times \CC^*$ that $\beta(\widetilde{\Gamma})\cap\alpha^{-1}(\overline{P}_r) = \varnothing$. We can also ensure that $d(\beta(\zeta),\zeta) < \epsilon/2$ for $\zeta \in K \subset \Delta_r \times\CC^*$. By the construction of $f$, after possibly shrinking $\rho$ we then have $\beta(\Gamma \cap (\overline{\Delta}_s \times \CC^*)) = \varnothing$. As discussed earlier, $\beta(z,w) \notin \alpha^{-1}(\overline{P}_r)$ for $|z| > s$, that is, for $(z,w) \in \Gamma \setminus \overline{\Delta}_s \times \CC^*$. Combining these results we see that $\beta(\Gamma) \cap \alpha^{-1}(\overline P_r) = \varnothing$.

The automorphism $\psi = \alpha \circ \beta$ will now satisfy conditions (a) and (b). Indeed, since $\beta(\Gamma) \cap \alpha^{-1}(\overline{P}_r) = \varnothing$, we have $\psi(\Gamma) \cap \overline{P}_r = \varnothing$, so that $\psi(\Gamma) \subset \CC\times\CC^* \setminus \overline{P}_r$. Let $\zeta \in K$, then $d(\beta(\zeta),\zeta) < \epsilon/2$, so that $\beta(\zeta) \in K'$ by our initial choice of $\epsilon$, and hence $d(\alpha(\beta(\zeta)),\beta(\zeta)) < \epsilon/2$. We have
\[
d(\psi(\zeta),\zeta) \le d(\alpha(\beta(\zeta)),\beta(\zeta)) + d(\beta(\zeta),\zeta) < \epsilon
\]
as required.
\end{proof}

The following lemma on exhaustions of embedded bordered Riemann surfaces in a connected Stein manifold $M$ is a generalisation of a similar result for embeddings of Riemann surfaces in $\CC^n$ as proved by Wold \cite{Wold:2006}. The proof follows directly from Wold's result, using the correspondence established in Lemma~\ref{convexity}.

\begin{lemma}
\label{compactexhaustion}
Let $M$ be a connected Stein manifold and $X \subset M$ be the interior of an embedded bordered Riemann surface $\overline X$ with non-compact boundary components $\partial_1,\dots,\partial_m$. Then there is an exhaustion $X_j$ of $X$ by $\OO(M)$-convex compact sets such that if $K \subset M \setminus \partial \overline X$ is an $\OO(M)$-convex compact set and $K \cap X \subset X_j$ for some $j$, then $K \cup X_j$ is $\OO(M)$-convex.
\end{lemma}
\begin{proof}
Let $\varphi : M \hookrightarrow \CC^n$ be an embedding. Then $\varphi(\overline X) \subset \CC^n$ is a bordered Riemann surface with unbounded boundary components $\varphi(\partial_1),\dots,\varphi(\partial_m)$. By the proof of Proposition 3.1 in \cite{Wold:2006} there exists an exhaustion $\widetilde{X}_j$ of $\varphi(X)$ by polynomially convex compact sets such that if $\widetilde{K} \subset \CC^n \setminus \varphi(\partial \overline X)$ is a polynomially convex compact set and $\widetilde{K} \cap \varphi(X) \subset \widetilde{X}_j$ for some $j$, then $\widetilde{K} \cup \widetilde{X}_j$ is polynomially convex. By Lemma~\ref{convexity}, the sets $X_j = \varphi^{-1}(\widetilde{X}_j)$ give an exhaustion of $X$ by $\OO(M)$-convex compact sets.

Now suppose we are given an $\OO(M)$-convex compact set $K \subset M \setminus \partial \overline X$ such that $K \cap X \subset X_j$ for some $j$. Then $\varphi(K) \subset \CC^n \setminus \varphi(\partial \overline X)$ is a polynomially convex compact set by Lemma~\ref{convexity} and $\varphi(K) \cap \varphi(X) \subset \widetilde{X}_j$. Thus $\varphi(K) \cup \widetilde{X}_j$ is polynomially convex and hence again by Lemma~\ref{convexity}, $K \cup X_j$ is $\OO(M)$-convex.
\end{proof}

The following proposition will allow us to construct a Fatou-Bieberbach domain in $\CC\times\CC^*$ with the properties detailed in the statement of Theorem~\ref{thm_Wold}. The proposition is a generalisation of a result due to Forstneri\v c \cite[Prop. 5.1]{Forstneric:1999} on the construction of Fatou-Bieberbach domains in $\CC^n$. The proof given there immediately yields the result in the more general situation stated below.

\begin{proposition}
\label{FatouBieberbach}
Let $S$ be a Stein manifold equipped with a distance function $d$ induced by a complete Riemannian metric on $S$. Let $D \subset S$ be an open connected set exhausted by a sequence of compact sets $K_0 \subset K_1 \subset \dots \subset \bigcup\limits_{j=0}^\infty K_j= D$ such that $K_{j-1} \subset K_j^\circ$ for each $j \in \NN$. Choose numbers $\epsilon_j$, $j\in\NN$, such that
\[
	0 < \epsilon_j < d(K_{j-1}, S \setminus K_j) \text{ for all } j \in \NN,
\]
and
\[
	\sum_{j=1}^\infty \epsilon_j < \infty\,.
\]

Suppose that $\Psi_j$, $j \in \NN$, are holomorphic automorphisms of $S$ satisfying
\[
d(\Psi_j(z),z) < \epsilon_j \text{ for all } z\in K_j \text{ and for all } j\in\NN\,.
\]
Set $\phi_m = \Psi_m\circ\Psi_{m-1}\circ\dots\circ\Psi_1$. Then there is an open set $\Omega \subset S$ such that the sequence $(\phi_m)$ converges locally uniformly on $\Omega$ to a biholomorphism $\phi : \Omega \to D$. In fact, we have $\Omega = \bigcup\limits_{m=1}^\infty \phi^{-1}_m(K_m)$.
\end{proposition}

We may now prove Theorem~\ref{thm_Wold}.

\begin{proof}[Proof of Theorem~\ref{thm_Wold}]
Let $\CC\times\CC^*$ be equipped with a distance function $d$ induced by a complete Riemannian metric (for example, pull back the Euclidean distance on $\CC^3$ via the embedding $\CC\times\CC^* \to \CC^3$ given by $(z,w)\mapsto (z,w,1/w)$). Consider the exhaustion $\varnothing = \overline{P}_0 \subset \overline{P}_1 \subset \overline{P}_2 \subset \dots$ of $\CC\times\CC^*$ by $\OO(\CC\times\CC^*)$-convex compact cylinders with integer radii. Choose a sequence $\epsilon_j > 0$, $j \in \NN$, such that ${\epsilon}_j < d(\overline{P}_{j-1},\CC\times\CC^* \setminus \overline{P}_j)$ and such that $\sum\limits_{j = 1}^\infty \epsilon_j < \infty$.

In the following we identify $\overline{X}$ with its image $\psi(\overline{X})$ in $\CC\times\CC^*$. Let $X_1 \subset X_2 \subset \dots$ be a compact exhaustion of $X$. By scaling $\overline{X}$ we may assume that $\overline{P}_2 \cap \partial\overline X = \varnothing$, due to the fact that $\partial\overline X$ satisfies the nice projection property. We begin by setting $\Psi_1 = \id \in \Aut(\CC\times\CC^*)$ so that the following conditions obviously hold:
\begin{itemize}
\item $d(\Psi_1(\zeta), \zeta) < \epsilon_1$ for all $\zeta \in \overline P_1$,
\item $d(\Psi_1(\zeta), \zeta) < \epsilon_1$ for all $\zeta \in X_1$, and
\item $\Psi_1(\partial\overline X) \subset \CC\times\CC^* \setminus \overline{P}_2$.
\end{itemize}
This completes the initial step of the induction.

We work through the second step of the induction explicitly. Consider the bordered Riemann surface $\Psi_1(\overline{X})$. Its boundary $\Psi_1( \partial\overline X)$ still satisfies the nice projection property since $\Psi_1 \in \Aut(\CC\times\CC^*)$. By Lemma~\ref{compactexhaustion} there exists an $\OO(\CC\times\CC^*)$-convex compact set $L \subset \Psi_1(X)$ such that $\overline{P}_2 \cap \Psi_1(X) \subset L$, $\Psi_1(X_2) \subset L$, and $\overline{P}_2 \cup L$ is $\OO(\CC\times\CC^*)$-convex. By Lemma~\ref{existauto} there exists $\Psi_2 \in \Aut(\CC\times\CC^*)$ such that $d(\Psi_2(\zeta),\zeta) < \epsilon_2$ for all $\zeta \in \overline P_2\cup L$ and $\Psi_2(\Psi_1( \partial\overline X)) \subset \CC\times\CC^* \setminus \overline{P}_3$. Since $\Psi_1(X_2) \subset L$, we have $d(\Psi_2(\zeta),\zeta) < \epsilon_2$ for all $\zeta\in \Psi_1(X_2)$ and thus the following conditions hold:
\begin{itemize}
\item $d(\Psi_2(\zeta),\zeta) < \epsilon_2$ for all $\zeta \in \overline P_2$,
\item $d(\Psi_2\circ\Psi_1(\zeta),\Psi_1(\zeta)) < \epsilon_2$ for all $\zeta \in X_2$, and
\item $\Psi_2 \circ \Psi_1( \partial\overline X)\subset \CC\times\CC^* \setminus \overline{P}_3$.
\end{itemize}

Continuing this process, considering at the $j$-th step the bordered Riemann surface $\Psi_{j-1}\circ\cdots\circ\Psi_1(\overline X)$, we obtain a sequence $(\Psi_j)$ of automorphisms of $\CC\times\CC^*$ such that for all $j \in \NN$:
\begin{enumerate}
\item \label{id_on_P} $d(\Psi_j(\zeta),\zeta) < \epsilon_{j}$ for all $\zeta \in \overline P_j$,
\item \label{Cauchy_on_X} $d(\Psi_j \circ \cdots \circ \Psi_1(\zeta), \Psi_{j-1}\circ\cdots\circ\Psi_1(\zeta)) < \epsilon_{j}$ for all $\zeta \in X_j$, and
\item \label{send_Gamma_away} $\Psi_j \circ \cdots \circ \Psi_1( \partial\overline X)\subset\CC\times\CC^*\setminus\overline{P}_{j+1}$.
\end{enumerate}

Setting $\phi_m = \Psi_m \circ \cdots \circ \Psi_1$, we apply Proposition~\ref{FatouBieberbach} to obtain an open set $\Omega= \bigcup\limits_{m=1}^\infty \phi_m^{-1}(\overline{P}_m) \subset \CC\times\CC^*$ such that $\phi = \lim\limits_{m\to\infty}\phi_m$ exists and converges uniformly on compact subsets of $\Omega$, and $\phi : \Omega \to \CC\times\CC^*$ is a biholomorphism. In fact, $\Omega$ equals the set of all points $\zeta \in \CC\times\CC^*$ such that the sequence $(\phi_j(\zeta))$ lies within a compact subset of $\CC\times\CC^*$. By (\ref{Cauchy_on_X}) we see that $X \subset \Omega$ and by (\ref{send_Gamma_away}) that $\partial \overline X \cap \Omega = \varnothing$, so that $\partial \overline X \subset \partial \Omega$. Thus $\phi$ gives a proper embedding of $X$ into $\CC\times\CC^*$.

Finally, if $K$ is a given compact subset of $X$, then it is clear that we may take $X_1$ sufficiently large so that $K \subset X_1$, while also taking $\epsilon_j$, $j \in \NN$, sufficiently small in order to ensure that $\phi|_K$ uniformly approximates the inclusion $K \hookrightarrow \CC\times\CC^*$ as closely as desired.
\end{proof}

\section{A strong Oka principle for circular domains}
\label{sec_strongOka}
\noindent As discussed in Section~\ref{sec_introduction}, a simple version of Gromov's Oka principle states that every continuous map from a Stein manifold into an elliptic manifold is homotopic to a holomorphic map. In this section we show that by restricting the class of domains of our maps to certain planar Riemann surfaces called \emph{circular domains}, and fixing the target as $\CC\times\CC^*$, a much stronger Oka property holds (Theorem~\ref{thm_strongOka}). We begin with the following definition.

\begin{definition}
\label{def_circularDomain}
A \emph{circular domain} is a domain $X \subset \CC$ consisting of the open unit disc $\Delta_1$ from which $n \ge 0$ closed, pairwise disjoint discs have been removed. Assuming the deleted discs have centres $a_j \in \Delta_1$ and radii $r_j > 0$, $j = 1,\dots,n$, we have
\[
	X = \Delta_1 \setminus \bigcup_{j = 1}^n (a_j + \overline \Delta_{r_j})
\]
with constraints $r_j + r_k < \lvert a_j - a_k \rvert$ for $j \neq k$, and $r_j < 1 - \lvert a_j \rvert$ for all $j = 1,\dots,n$. Note that our circular domains are not permitted to have punctures.
\end{definition}

By the Koebe uniformisation theorem \cite[Ch.\ V, \textsection 6, Thm. 2]{Goluzin:1969}, every finitely connected open subset of the Riemann sphere is biholomorphic to the Riemann sphere with a finite number of points and pairwise disjoint closed discs removed. Thus any finitely connected, bounded planar domain without isolated boundary points is biholomorphic to a circular domain. Note that this is the same class of Riemann surfaces as considered by Globevnik and Stens\o nes in \cite{Globevnik:1995}.

We now state the main result of the paper, a strong Oka property for circular domains.

\begin{theorem}
\label{thm_strongOka}
Let $X$ be a circular domain. Then every continuous map $X \to \CC\times\CC^*$ is homotopic to a proper holomorphic embedding $X \to \CC\times\CC^*$.
\end{theorem}

In order to prove Theorem~\ref{thm_strongOka} we will prove the following result on embeddings of Riemann surfaces into $\CC\times\CC^*$. Note that if $\overline X$ is a compact bordered Riemann surface then $\partial\overline X$ has finitely many components $\gamma_1,\dots,\gamma_m$, each a compact, smooth one-dimensional manifold diffeomorphic to the circle.

\begin{theorem}
\label{thm_embedRiemannSurface}
Let $\overline X$ be a compact bordered Riemann surface and $f : \overline X \to \CC\times\CC^*$ be an embedding. Then $f$ can be approximated, uniformly on compact subsets of $X$, by embeddings $X \to \CC\times\CC^*$ that are homotopic to $f|_X$.
\end{theorem}

This result parallels that of Forstneri\v c and Wold in \cite{Forstneric:2009} for embeddings of certain Riemann surfaces into $\CC^2$, except that we prove an additional fact relating to the homotopy classes of embeddings produced by the theorem (the corresponding fact for embeddings into $\CC^2$ is of course trivial). In fact, the proof given in \cite{Forstneric:2009} adapts immediately to the situtation of embeddings into $\CC\times\CC^*$, so that we only outline the existing proof while taking care to demonstrate that the embeddings obtained are homotopic to the original map.

Note that Theorem~\ref{thm_embedRiemannSurface} makes no statement as to which bordered Riemann surfaces can be embedded into $\CC\times\CC^*$. While it is possible to explicitly construct embeddings of some classes of bordered Riemann surfaces into $\CC\times\CC^*$, doing so in general is a difficult outstanding problem, as it is for embeddings into $\CC^2$.

We remark that Theorem~\ref{thm_embedRiemannSurface} is in a similar spirit as a result of Drinovec-Drnov\v sek and Forstneri\v c \cite{Drinovec-Drnovsek:2010} where it is proved that a continuous map $f : \overline X \to Y$, holomorphic in $X$, where $X$ is a smoothly bounded, relatively compact, strongly pseudoconvex domain in a Stein manifold $S$, and $Y$ is a Stein manifold, can be approximated uniformly on compact subsets of $X$ by proper holomorphic embeddings $X \to Y$, provided $\dim Y \ge 2 \dim X + 1$.

The proof of Theorem~\ref{thm_embedRiemannSurface} is broken into three parts, the first of which requires the concept of an \emph{exposed point} \cite[Def. 4.1]{Forstneric:2009} of an embedded bordered Riemann surface.

\begin{definition}
Let $f : \overline X \to \CC\times\CC^*$ be an embedding of the bordered Riemann surface $\overline X$ into $\CC\times\CC^*$, and let $p \in \overline X$. We say $p$ is an \emph{$f$-exposed point} (or that $f(p)$ is an \emph{exposed point}) if the complex line
\[
	\pi_2^{-1}(\pi_2(f(p))) = \CC\times\{\pi_2(f(p))\}
\]
intersects $f(\overline X)$ precisely at $f(p)$ and the intersection is transverse.
\end{definition}

\begin{theorem}
\label{thm_exposeBoundaryPoints}
Let $\overline X$ be a compact bordered Riemann surface. For each component $\gamma_j$ of $\partial \overline X$, $j = 1,\dots,m$, let $a_j \in \gamma_j$ and let $U_j$ be an open neighbourhood of $a_j$ in $\overline X$. Let $f :\overline X \to \CC\times\CC^*$ be an embedding. Then $f$ can be approximated uniformly on $\overline X \setminus \bigcup\limits_{j=1}^m U_j$ by embeddings $\overline X \to \CC\times\CC^*$ homotopic to $f$, such that the points $a_1,\dots,a_m$ are all $f$-exposed.
\end{theorem}

\begin{proof}[Proof (sketch)]
In \cite[Thm. 4.2]{Forstneric:2009} the authors begin by choosing for each $j = 1,\dots,m$ a point $a_j$ in the boundary component $\gamma_j$ of $\overline X$, together with a smoothly embedded curve $\lambda_j$ in $\CC^2$, starting at $p_j = f(a_j)$ and ending at a point $q_j$, such that $\lambda_j$ meets $f(\overline X)$ transversally only at $p_j$, the $\lambda_j$ are pairwise disjoint, and the endpoints $q_j$ are exposed for the set $f(\overline X) \cup \bigcup\limits_{j=1}^m \lambda_j$. They then use this information to construct an embedding $F : \overline X \to \CC^2$ that approximates $f$ uniformly on $\overline X \setminus \bigcup\limits_{j=1}^m U_j$ such that each point $a_j$ is $F$-exposed, where each $U_j$ is an arbitrarily small open neighbourhood in $\overline X$ of the boundary point $a_j$. From the construction in \cite{Forstneric:2009} it is clear that the image $F(\overline X)$ can be made to lie within an arbitrary open neighbourhood $V$ of the set $f(\overline X) \cup \bigcup\limits_{j=1}^m \lambda_j$. Thus in our case, beginning with an embedding $f : \overline X \to \CC\times\CC^*$, we may apply this construction first in $\CC^2$ with the curves $\lambda_j$ chosen to lie in $\CC\times\CC^*$ and then approximate $f$ sufficiently well by $F$ to ensure that $F(\overline X)$ lies entirely in $\CC\times\CC^*$.

It remains to be shown that $F$ and $f$ are homotopic as maps $\overline X \to \CC\times\CC^*$. If we choose each $U_j$ to be a sufficiently small contractible open neighbourhood of $a_j \in \partial \overline X$ then it is clear that each $\overline U_j$ deformation retracts within $\overline X$ to $\overline U_j \setminus U_j$, where the closure of $U_j$ is taken in $\overline X$. We then see that $\overline X$ deformation retracts to $\overline X \setminus \bigcup\limits_{j=1}^m U_j$. Letting $F$ approximate $f$ sufficiently well on $\overline X \setminus \bigcup\limits_{j=1}^m U_j$, a convex linear combination will deform $f\big|({\overline X \setminus \bigcup\limits_{j=1}^m U_j})$ to $F\big|({\overline X \setminus \bigcup\limits_{j=1}^m U_j})$ without passing through the missing line in $\CC\times\CC^*$. Thus $f : \overline X \to \CC\times\CC^*$ and $F : \overline X \to \CC\times\CC^*$ are homotopic.
\end{proof}

Once we have constructed an embedding of $\overline X$ that exposes a point $a_j$ in each component of $\partial\overline X$, the following result shows that we can send the points $a_j$ to infinity in $\CC\times\CC^*$ in such a way as to obtain an embedding of the bordered Riemann surface $\overline X \setminus \{a_1,\dots,a_m\}$ so that the image of its (now non-compact) boundary components satisfy the nice projection property, while preserving the homotopy class of the embedding.

\begin{theorem}
\label{thm_ensureNiceProjectionProperty}
Let $f : \overline X \to \CC\times\CC^*$ be an embedding of the compact bordered Riemann surface $\overline X$ such that every component $\gamma_j$, $j = 1,\dots, m$, of $\partial \overline X$ contains an $f$-exposed point $a_j$. Then $f$ can be approximated uniformly on compact subsets of $X$ by embeddings $\overline X \setminus \{a_1,\dots,a_m\} \to \CC\times\CC^*$, whose restrictions to $X$ are homotopic to $f|_X$, and such that the image of the boundary components of $\overline X \setminus \{a_1,\dots,a_m\}$ under each approximating embedding satisfy the nice projection property.
\end{theorem}

\begin{proof}[Proof (sketch)]
Following \cite[Thm. 5.1]{Forstneric:2009} we define a rational shear $g$ by
\[
	g(z,w) = (z + \sum_{j=1}^m \frac{\alpha_j}{w - \pi_2(f(a_j))},w)\,.
\]
In \cite{Forstneric:2009} it is explained (with reference also to \cite{Wold:2006}) that the arguments of $\alpha_1,\dots,\alpha_m \in \CC^*$ may be chosen so that the image under $g \circ f$ of the boundary components of $\overline X \setminus \{a_1,\dots,a_m\}$ satisfy the nice projection property. Choosing $\alpha_1,\dots,\alpha_m$ sufficiently small, we can ensure that the map $g\circ f : \overline X \setminus \{a_1,\dots,a_m\} \to \CC\times\CC^*$ approximates $f$ uniformly on a given compact subset of $X$. As $g$ is an invertible holomorphic map on $f(\overline X) \setminus \{f(a_1),\dots,f(a_m)\}$, $g\circ f$ gives an embedding $\overline X \setminus \{a_1,\dots,a_m\} \to \CC\times\CC^*$ with the desired properties, and $(g \circ f)|_X$ is clearly homotopic to $f|_X$ by a convex linear combination in the first coordinate.
\end{proof}

Finally, we prove the following lemma on the homotopy class of the embedding $\sigma : X \to \CC\times\CC^*$ given by the Wold embedding theorem (Theorem~\ref{thm_Wold}). Combining Theorems~\ref{thm_Wold},~\ref{thm_exposeBoundaryPoints} and \ref{thm_ensureNiceProjectionProperty} with the lemma then immediately proves Theorem~\ref{thm_embedRiemannSurface}.

\begin{lemma}
\label{lem_WoldEmbeddingHomotopyClass}
Let $X$ be an open Riemann surface that is the interior of a bordered Riemann surface $\overline{X}$ whose boundary components are non-compact and finite in number. Let $\psi : \overline X \to \CC\times\CC^*$ be an embedding such that $\psi(\partial\overline{X})$ has the nice projection property. Then the embedding $\sigma : X \to \CC\times\CC^*$ given by Theorem~\ref{thm_Wold} is homotopic to $\psi|_X$.
\end{lemma}

\begin{proof}
We will henceforth write $\psi$ for the restriction $\psi|_X : X \to \CC\times\CC^*$. Recall that in Theorem~\ref{thm_Wold} we construct a Fatou-Bieberbach domain $\Omega$ in $\CC\times\CC^*$ and a biholomorphism $\phi : \Omega \to \CC\times\CC^*$ such that $\psi(X)\subset \Omega$. We may therefore factorise $\psi$ as $\iota \circ \widetilde \psi$, where $\widetilde \psi : X \to \Omega$ is given by restricting the target of $\psi$ to $\Omega$, and $\iota : \Omega \hookrightarrow \CC\times\CC^*$ is the inclusion. From Theorem~\ref{thm_Wold} the embedding $\sigma : X \to \CC\times\CC^*$ is given by $\sigma = \phi \circ \psi = \phi \circ \widetilde\psi$. Comparing the factorisations of $\sigma$ and $\psi$ we see that in order to show these two maps are homotopic it suffices to prove that $\phi$ and $\iota$ are homotopic maps from $\Omega$ to $\CC\times\CC^*$.

Both spaces $\Omega$ and $\CC\times\CC^*$ are homotopy equivalent to the circle $S^1$. Using the homotopy equivalences we obtain induced maps $\tilde\phi, \tilde\iota : S^1 \to S^1$, and these two maps will be homotopic if and only if the maps $\phi$ and $\iota$ are homotopic. However, the homotopy class of a map from $S^1$ to $S^1$ is purely determined by its degree, so that $\tilde\phi$ and $\tilde\iota$ are homotopic if and only if they have the same degree. Transferring this condition back to the original maps we see that $\phi$ and $\iota$ are homotopic if and only if they take a loop that generates $\pi_1(\Omega)$ to homotopic loops in $\pi_1(\CC\times\CC^*)$.

We take the loop $\alpha : [0,1] \to \CC\times\CC^*$, $\alpha(t) = (0,e^{2\pi it})$, which generates $\pi_1(\CC\times\CC^*)$. Using the biholomorphism $\phi$ we pull this back to the loop $\phi^{-1} \circ \alpha$ in $\Omega$, which generates $\pi_1(\Omega)$. Clearly, the map $\phi$ applied to $\phi^{-1}\circ\alpha$ gives the loop $\alpha$ in $\CC\times\CC^*$. We now show that $\iota$ also takes the loop $\phi^{-1}\circ\alpha$ to a loop homotopic to $\alpha$, which will prove the result. That is, we show that $\phi^{-1}\circ\alpha$, considered as a loop in $\CC\times\CC^*$, is homotopic to $\alpha$.

Recall the definition of $\phi = \lim\limits_{m\to\infty} \phi_m$, where each $\phi_m \in \Aut(\CC\times\CC^*)$, and the convergence is uniform on compact subsets of $\Omega \subset \CC\times\CC^*$. By \cite[Thm. 5.2]{Dixon:1986} we also have $\phi^{-1} = \lim\limits_{m\to\infty}\phi_m^{-1}$, uniformly on compact subsets of $\CC\times\CC^*$. Choose $\epsilon > 0$ sufficiently small so that the open $\epsilon$-neighbourhood $U$ of the image of $\phi^{-1}\circ\alpha$ is relatively compact in $\CC\times\CC^*$, and then take $N \in \NN$ so that for all $m \ge N$,
\[
	\lVert \phi^{-1}\circ\alpha - \phi^{-1}_m\circ\alpha\rVert_{[0,1]} < \epsilon\,,
\]
where $\lVert \cdot \rVert$ denotes the Euclidean norm.

For $m \ge N$, each loop $\phi^{-1}_m\circ\alpha$ can be deformed within $U$ by a convex linear combination to the loop $\phi^{-1}\circ\alpha$ so that $\phi^{-1}_m\circ\alpha$ represents the same class as $\phi^{-1}\circ\alpha$ in $\pi_1(\CC\times\CC^*)$. However, by construction in Lemma~\ref{existauto} and Theorem~\ref{thm_Wold}, each $\phi_m$, and hence each $\phi^{-1}_m$, is homotopic to the identity map on $\CC\times\CC^*$, so that $\phi^{-1}\circ\alpha$ is homotopic to $\alpha$, as required.
\end{proof}

Note that in the above proof we established that $\iota : \Omega \hookrightarrow \CC\times\CC^*$ induces an isomorphism of the fundamental groups of the Fatou-Bieberbach domain $\Omega$ and $\CC\times\CC^*$. We can interpret this as saying that the Fatou-Bieberbach domain $\Omega$ is untwisted inside $\CC\times\CC^*$. It is unclear to me whether another method of constructing a Fatou-Bieberbach domain in $\CC\times\CC^*$ could give rise to a twisted Fatou-Bieberbach domain, that is, one for which the inclusion $\Omega \hookrightarrow \CC\times\CC^*$ is not surjective on fundamental groups.

We may now prove Theorem~\ref{thm_strongOka}.

\begin{proof}[Proof of Theorem~\ref{thm_strongOka}]
Suppose first that $X$ is the $n$-connected circular domain
\[
	X = \Delta_1 \setminus \bigcup_{j=1}^n (a_j + \overline\Delta_{r_j})
\]
where $n > 0$ and the $r_j$ and $a_j$ satisfy conditions as in Definition~\ref{def_circularDomain}, and let the compact bordered Riemann surface $\overline X$ be its closure in $\CC$. Let $f : X \to \CC\times\CC^*$ be a continuous map. For a sufficiently small choice of $\delta > 0$ the loops $\beta_j(t) = a_j + (r_j + \delta) e^{2\pi it}$, $t \in [0,1]$, $j = 1, \dots, n$, lie in $X$ with each $\beta_j$ only encircling the hole given by deleting the disc $\overline\Delta_{r_j}$. If we pick an arbitrary basepoint $x_0 \in X$ and for each $j$ let $\tilde\beta_j$ be the conjugation of $\beta_j$ by a path in $X$ from $x_0$ to the basepoint $a_j + r_j + \delta$ of $\beta_j$, then $\pi_1(X,x_0)$ equals the free group generated by the loops $\tilde\beta_1,\dots,\tilde\beta_n$. 

Suppose $g, h : X \to \CC\times\CC^*$ are continuous. Since $X$ is homotopy equivalent to a bouquet of circles and $\CC\times\CC^*$ to a single circle, and maps between such spaces are determined up to homotopy by the number of times they wind each circle in the source about the target circle, we see that $g$ and $h$ are homotopic precisely when the loops $g\circ\beta_j$ and $h\circ\beta_j$ are homotopic in $\CC\times\CC^*$ for all $j = 1,\dots,n$. In order to prove Theorem~\ref{thm_strongOka} we now give the construction of an embedding $\tilde f$ that winds each loop $\beta_j$ about the missing line in $\CC\times\CC^*$ the same number of times as does $f$.

Define integers $k_j$, $j = 1,\dots,n$, so that in $\pi_1(\CC\times\CC^*)$ we have $[f\circ\beta_j] = [\alpha]^{k_j}$, where $\alpha(t) = (0,e^{2\pi it})$, $t \in [0,1]$, is a generator of $\pi_1(\CC\times\CC^*)$. Now take the holomorphic function $q(z) = \prod\limits_{j=1}^n (z - a_j)^{k_j}$ on $\overline X$, and consider the map $p : \overline X \to \CC\times\CC^*$ given by $p(z) = (z,q(z))$. This is clearly an embedding of $\overline X$ into $\CC\times\CC^*$ such that the image of each loop $\beta_j$ under $p$ winds $k_j$ times about the missing line. Thus $f$ and $p|_X$ are homotopic.

We now apply Theorem~\ref{thm_embedRiemannSurface} to give an embedding $\tilde f : X \to \CC\times\CC^*$ that is homotopic to $p|_X$, and therefore also homotopic to $f$.

In the case that $X$ is simply connected, we have $X = \Delta_1$ and every continuous map $f : X \to \CC\times\CC^*$ is null-homotopic. We take the embedding $p : \overline X \to \CC\times\CC^*$ given by $p(z) = (0,2+z)$, and then $p(1) \in p(\partial \overline X)$ is an exposed point. Applying Theorem~\ref{thm_ensureNiceProjectionProperty} and then Theorem~\ref{thm_Wold} we obtain an embedding $\tilde f : X \to \CC\times\CC^*$ that is homotopic to $f$, thereby completing the proof.
\end{proof}

A number of special cases of embeddings of open Riemann surfaces into $\CC\times\CC^*$ are worth discussing explicitly. Embeddings of the open disc and proper open annuli are of historical interest, and applying Theorem~\ref{thm_strongOka} we see that these can all be embedded into $\CC\times\CC^*$. In the case of a proper annulus, we can ensure the embedding winds the annulus around the missing line in $\CC\times\CC^*$ any number of times in either direction. Taking an embedding with degree 1 we then see that the annulus can be acyclically embedded into $\CC\times\CC^*$, that is, embedded such that the embedding induces a homotopy equivalence between the annulus and $\CC\times\CC^*$. The existence of acyclic embeddings into $\CC\times\CC^*$ is of interest for reasons discussed in Section~\ref{sec_introduction}.

On the question of which other open Riemann surfaces may be acyclically embedded into $\CC\times\CC^*$, it is immediate that $\CC^*$ embeds acyclically, with the punctured disc the only case remaining to be considered. However, if we map the punctured closed unit disc $\overline \Delta_1 \setminus \{0\}$ into $\CC\times\CC^*$ by $z \mapsto (0,z)$ then this is a proper embedding under which every point in the boundary circle is exposed. We may therefore send the point $(0,1)$ to infinity by a rational shear as in Theorem~\ref{thm_ensureNiceProjectionProperty}, and then apply Theorem~\ref{thm_Wold} to obtain an acyclic embedding of the punctured open disc into $\CC\times\CC^*$ (note also that the embedding of $\overline \Delta_1 \setminus \{0\}$ given by $z \mapsto (1/z,2+z)$ similarly gives rise to a null-homotopic embedding of the punctured open disc into $\CC\times\CC^*$). We therefore have the following corollary:

\begin{corollary}
\label{cor_acyclicEmbedding}
Every open Riemann surface with non-trivial abelian fundamental group embeds acyclically into $\CC\times\CC^*$.
\end{corollary}

It is well known that the fundamental group of every open Riemann surface is freely generated, so the only other open Riemann surfaces with abelian fundamental group are $\CC$ and the open disc. In the case of $\CC$, an acyclic embedding into $\CC^2$ exists trivially, while for the disc an embedding into $\CC^2$ exists by the results of Kasahara and Nishino \cite{Stehle:1972}, and this embedding is trivially acyclic. Since $\CC^2$ is an elliptic Stein manifold, we have:

\begin{corollary}
\label{cor_acyclicEmbedding2}
Every open Riemann surface with abelian fundamental group embeds acyclically into a 2-dimensional elliptic Stein manifold.
\end{corollary}

The results in this paper focus on circular domains without punctures, that is, proper open subsets of the Riemann sphere whose boundary components are all circles. If we attempt a systematic study of domains with punctures the situation becomes more complicated, and it is unclear whether a strong Oka principle holds for such domains. One difficulty is that the Wold embedding theorem does not yield properness at punctures in the domain, as it does near circular boundary components of the domain (after sending an exposed point in the boundary component to infinity). Instead, properness at the punctures needs to be built into the initial embedding $\psi : \overline X \to \CC\times\CC^*$, together with appropriate winding numbers about each puncture. I plan to further investigate embeddings of finitely connected circular domains with punctures in the near future.

\appendix
\section*{Appendix: The Anders\'en-Lempert theorem}
\label{sec_AndersenLempert}
\noindent As discussed in Section~\ref{sec_Wold}, Anders\'en and Lempert \cite{Andersen:1992} were the first to show that an injective holomorphic map $\Phi : \Omega \to \CC^n$ from a star-shaped domain $\Omega \subset \CC^n$, $n\ge 2$, onto a Runge domain $\Phi(\Omega)$ can be approximated uniformly on compact subsets of $\Omega$ by automorphisms of $\CC^n$. In \cite{Forstneric:1993,Forstneric:1994}, Forstneri\v c and Rosay generalised the argument of Anders\'en and Lempert, giving a theorem on the approximation of parametrised families of injective holomorphic maps by automorphisms. In the same papers Forstneri\v c and Rosay also proved a more general result on approximation by automorphisms in a neighbourhood of a polynomially convex compact set $K \subset \CC^n$. These results, together with stronger theorems by Forstneri\v c and L\o w \cite{Forstneric:1997}, and Forstneri\v c, L\o w and \O vrelid \cite{Forstneric:2001}, are referred to as \emph{Anders\'en-Lempert theorems} for $\CC^n$.

In this appendix we prove, in full detail, generalisations of the Anders\'en-Lempert theorems of Forstneri\v c and Rosay to Stein manifolds with the density property. Although the arguments follow closely those given in \cite{Forstneric:1993,Forstneric:1994}, with some modifications required for the more general setting, I believe these results have not appeared in the literature and so provide full details for the benefit of the reader. Recall that a smooth vector field on a smooth manifold $X$ is said to be \emph{$\RR$-complete} if its maximal domain equals $\RR\times X$ (Definition~\ref{def_completeVectorField}).

\begin{definition}
Let $X$ be a complex manifold and let $\mathfrak{X}(X)$ be the Lie algebra of holomorphic vector fields on $X$. We say $X$ has the \emph{density property} if the Lie algebra generated by the $\RR$-complete holomorphic vector fields on $X$ is dense in $\mathfrak{X}(X)$ in the compact-open topology.
\end{definition}


The following main result is a generalisation of Theorem 2.1 in \cite{Forstneric:1993}, and appears as Theorem~\ref{thm_AndersenLempertHolConvex} in Section~\ref{sec_Wold}. We recall the definition of a \emph{$\Cont^k$-isotopy of injective holomorphic maps} (Definition~\ref{def_isotopy}).

\begin{theorem}
\label{thm_AndersenLempertHolConvex_app}
Let $X$ be a Stein manifold with the density property. Let $\Omega \subset X$ be an open set and $\Phi_t : \Omega \to X$ be a $\Cont^1$-isotopy of injective holomorphic maps such that $\Phi_0$ is the inclusion of $\Omega$ into $X$. Suppose $K \subset \Omega$ is a compact set such that $K_t = \Phi_t(K)$ is $\mathscr{O}(X)$-convex for every $t \in [0,1]$. Then $\Phi_1$ can be uniformly approximated on $K$ by holomorphic automorphisms of $X$ with respect to any Riemannian distance function on $X$.
\end{theorem}

We first prove the following slightly simpler result, which is a generalisation of Theorem 1.1 in \cite{Forstneric:1993}. Note that, for us, Runge sets are always taken to be Stein.

\begin{theorem}
\label{AndersenLempert}
Let $X$ be a Stein manifold with the density property. Let $\Omega \subset X$ be an open set and $\Phi_t : \Omega \to X$ be a $\Cont^1$-isotopy of injective holomorphic maps such that $\Phi_0$ is the inclusion of $\Omega$ into $X$. Suppose that the set $\Omega_t = \Phi_t(\Omega)$ is Runge in $X$ for every $t \in [0,1]$. Then $\Phi_1$ can be uniformly approximated on compact subsets of $\Omega$ by holomorphic automorphisms of $X$.
\end{theorem}

We need a number of definitions and results in order to prove this theorem.

\begin{definition}
Let $V$ be a smooth vector field on a smooth manifold $X$ and $(t,x) \mapsto A_t(x)$ be a continuous map from an open set in $[0,\infty) \times X$ containing $\{0\}\times X$ to $X$ such that the $t$-derivative of $A_t$ exists and is continuous. We say that $A$ is an \emph{algorithm} for $V$ if for all $x \in X$ we have
\[
A_0(x) = x \quad \mathrm{and} \quad \frac{\partial}{\partial t}\bigg|_{t=0}A_t(x) = V(x)\,.
\]
\end{definition}

The proof of the following result can be found in \cite[Thm. 2.1.26]{Abraham:1978}.

\begin{theorem}
\label{algorithm}
Let $V$ be a smooth vector field with flow $\phi_t$ on a smooth manifold $X$. Let $\Omega \subset \RR \times X$ be the maximal domain of $V$ and let $\Omega_+ = \Omega \cap ([0,\infty) \times X)$. If $A$ is an algorithm for $V$ then for all $(t,x) \in \Omega_+$ the $n$-th iterate $A^n_{t/n}(x)$ of the map $A_{t/n}$ is defined for sufficiently large $n \in \NN$ (depending on $x$ and $t$), and we have
\[
\lim_{n\to\infty} A^n_{t/n}(x) = \phi_t(x)\,.
\]
The convergence is uniform on compact sets in $\Omega_+$.
\end{theorem}

Applying Theorem~\ref{algorithm} to appropriate choices of algorithms we obtain the following proposition (see \cite[Cor. 2.1.27]{Abraham:1978} and \cite[Prop. 4.2.34]{Abraham:1988}).

\begin{proposition}
Let $V$ and $W$ be smooth vector fields with flows $\phi_t$ and $\psi_t$. Then
\begin{enumerate}[{\normalfont (i)}]
\item $\phi_t \circ \psi_t$ is an algorithm for $V + W$.
\item $\psi_{-\sqrt{t}}\circ\phi_{-\sqrt{t}}\circ\psi_{\sqrt{t}}\circ\phi_{\sqrt{t}}$ is an algorithm for $[V,W]$.
\end{enumerate}
\end{proposition}

By repeated application of the preceding two results, we have:

\begin{corollary}
\label{cor_approxbyauto}
Let $V_1, \dots, V_m$ be $\RR$-complete holomorphic vector fields on a complex manifold $X$. Let $V$ be a holomorphic vector field on $X$ that is in the Lie algebra generated by $V_1, \dots, V_m$. Assume that $K\subset X$ is a compact set and $t_0 > 0$ is such that the flow $\phi_t(x)$ of $V$ exists for all $x \in K$ and for all $t \in [0, t_0]$. Then $\phi_{t_0}$ is a uniform limit on $K$ of a sequence of compositions of forward-time maps of the vector fields $V_1,\dots,V_m$. In particular, since the forward-time maps of $V_1,\dots,V_m$ are automorphisms of $X$, $\phi_{t_0}$ can be uniformly approximated on $K$ by automorphisms of $X$.
\end{corollary}

In the proof of Theorem~\ref{AndersenLempert} we will approximate a vector field $V$ uniformly on a given set by another vector field $W$, and we will need to know that the flows of nearby points $x$ and $y$ along $V$ and $W$ respectively can be made to remain close for finite time. This is stated precisely in the following lemma. We sketch a proof for the convenience of the reader.

\begin{lemma}
\label{lem_flowapprox}
Let $X$ be a Riemannian manifold, $\Omega \subset X$ be a relatively compact open subset, and $V$ be a smooth vector field on $\Omega$ with flow $\phi_t$. Let $K \subset \Omega$ be a compact set and let $t_0 > 0$ be such that for all $x \in K$ the flow $\phi_t(x)$ of $x$ along $V$ is defined for all $t \in [0,t_0]$. Then there exists $\eta > 0$ such that if $W$ is a smooth vector field on $\Omega$ with flow $\psi_t$ satisfying $\lVert V - W \rVert_{L^\infty(\Omega)} < \eta$, the flow $\psi_t(x)$ of $x$ along $W$ is defined for all $t \in [0,t_0]$, for all $x \in K$.

Furthermore, given $\epsilon > 0$, and after possibly shrinking $\eta$, there exists $\delta > 0$ such that for all $x,y \in K$ with $d(x,y) < \delta$, we have
\[
d(\phi_{t_0}(x),\psi_{t_0}(y))<\epsilon\,.
\]
\end{lemma}
\begin{proof}[Proof (sketch)]
In the case that $X = \RR^N$, the lemma is true by a standard argument involving Gr\"onwall's inequality. To see that the result holds for an arbitrary Riemannian manifold $X$, first use Whitney's embedding theorem to obtain a proper embedding $\sigma: X \to \RR^N$ for some $N$. Using a tubular neighbourhood of $\sigma(X)$ we may extend the smooth vector fields $\sigma_*(V)$ and $\sigma_*(W)$ on $\sigma(\Omega)$ to a relatively compact open set in $\RR^n$ that contains $\sigma(K)$. Applying the lemma for $\RR^N$ and using the relative compactness of $\Omega$ then proves the result.
\end{proof}

We can now prove Theorem~\ref{AndersenLempert}.

\begin{proof}[Proof of Theorem~\ref{AndersenLempert}]
Let $\widetilde{\Omega}$ denote the trace of the isotopy $\Phi_t$:
\[
\widetilde{\Omega} = \{(t,x) : t \in [0,1], x \in \Omega_t\} \subset \RR \times X\,.
\]
By differentiating $\Phi_t$ with respect to $t$ we obtain the following time-dependent vector field, whose continuity follows by a result of Dixon and Esterle \cite[Thm. 5.2]{Dixon:1986}:
\[
V(t,x) = \dot\Phi_t(\Phi_t^{-1}(x))\,,\quad t \in [0,1]\,,\ x \in \Omega_t\,.
\]
For fixed $t \in [0,1]$, $V_t = V(t,\cdot)$ is a holomorphic vector field defined on $\Omega_t \subset X$. We may also define an autonomous vector field $\widetilde{V}$ on $\widetilde{\Omega} \subset \RR \times X$ corresponding to $V$ by
\[
\widetilde{V}(t,x) = (1, V(t,x))\in T\RR\times TX\,,\quad (t,x) \in \widetilde{\Omega}\,,
\]
with flow $\widetilde{\Phi}_s(t,x) = (t + s, \Phi_{t+s}(\Phi_t^{-1}(x)))$. By taking $t = 0, s = 1$ we see that $\widetilde{\Phi}_1(0,x) = (1, \Phi_1(x))$, so that we may approximate the map $\Phi_1$ by approximating the $s = 1$ forward-time map of $\{0\} \times \Omega$ along $\widetilde{V}$. We begin by defining an algorithm for $\widetilde{V}$.

Given $(t,x) \in \widetilde{\Omega}$, let $\widetilde{A}_s(t,x) = (t+s, \Psi_s^t(x))$, where $\Psi^t$ is the flow of the autonomous holomorphic vector field $V_t = V(t,\cdot)$  defined on $\Omega_t \subset X$. Thus $\Psi^t_s(x)$ is the flow of $x \in \Omega_t$ along the holomorphic vector field $V_t$ for time $s$, and the domain $U$ of $\widetilde{A}_s$ is then an open subset of $[0,\infty) \times \widetilde{\Omega}$ containing $\{0\} \times \widetilde{\Omega}$. We have $\frac{\partial}{\partial s}\big|_{s=0} \widetilde{A}_s(t,x) = (1,V(t,x)) = \widetilde{V}(t,x)$, and $\widetilde{A}_s$ and its $s$-derivative are continuous, making $\widetilde{A}_s$ an algorithm for $\widetilde{V}$. Let $A_s(t,x) : U \to X$ be defined by $A_s(t,x) = \Psi_s^t(x)$, the projection of $\widetilde{A}$ onto the $X$-component.

Applying Theorem~\ref{algorithm} to $\widetilde{A}$ for the compact set $\{0\}\times\widetilde{K}$, where $\widetilde{K}\subset\Omega$ is a slightly larger compact set such that $K \subset \widetilde{K}^\circ$, and then projecting onto the $X$-component we obtain $n \in \NN$ such that, for all $x \in \widetilde{K}$,
\[
d(\Phi_1(x),A^n_{1/n}(0,x)) < \epsilon/2\,,
\]
where
\[
A^n_{1/n} = \Psi^{1-1/n}_{1/n}\circ \Psi^{1-2/n}_{1/n} \circ \cdots \circ \Psi^{1/n}_{1/n} \circ \Psi^0_{1/n}
\]
is the composition of the time $1/n$ flows of the autonomous holomorphic vector fields $V_0, V_{1/n}, \dots, V_{1-1/n}$, defined on $\Omega_0, \Omega_{1/n}, \dots, \Omega_{1-1/n}$ respectively. At this point in the proof we simplify our notation by rescaling the time coordinate by a factor of $n$, so that $t \in [0, n]$. We are thus now interested in approximating the map $\Phi_n$. After the rescaling, we have 
\begin{equation}
\label{eqn_approxPhi1}
d(\Phi_n(x),A^n_{1}(0,x)) < \epsilon/2\,,\  x \in \widetilde{K}\,,
\end{equation}
where
\[
A^n_{1} = \Psi^{n-1}_{1}\circ \Psi^{n-2}_{1} \circ \cdots \circ \Psi^{1}_{1} \circ \Psi^0_{1}
\]
is the composition of the $t=1$ flows of the autonomous vector fields $V_0, \dots, V_{n-1}$, defined on $\Omega_0, \Omega_{1}, \dots, \Omega_{n-1}$ respectively.

If we now set
\[
\widetilde{K}_0 = \widetilde{K}, \widetilde{K}_{1} = \Psi^0_{1}(\widetilde{K}_0), \dots, \widetilde{K}_{n-1} = \Psi^{n-2}_{1}(\widetilde{K}_{n-2})\,,
\]
and similarly define $K_0, K_{1}, \dots, K_{n-1}$, then $K_{j}$ and $\widetilde{K}_{j}$ are compact subsets of $\Omega_{j}$ such that $K_{j} \subset \widetilde{K}_{j}^\circ$, for all $j = 0, \dots, n-1$. Note that we have subdivided the interval $[0,n]$ into $n$ subintervals $[0,1],\dots,[n-1,n]$, each of length $1$. We proceed by examining the final subinterval $[n-1, n]$.

On the final subinterval $[n-1, n]$ we approximate the flow of the autonomous holomorphic vector field $V_{n-1}$, defined on $\Omega_{n-1}$. We first choose an open set $Z_{n-1} \subset \subset \Omega_{n-1}$ that contains the flows of all points in $\widetilde{K}_{n-1}$ along $V_{n-1}$ up to time $1$. Using Lemma~\ref{lem_flowapprox} we choose $\eta_{n-1}, \delta_{n-1} > 0$ sufficiently small so that if $W_{n-1} \in \mathfrak{X}(\Omega_{n-1})$ satisfies $\lVert V_{n-1} - W_{n-1} \rVert_{L^\infty(Z_{n-1})} < \eta_{n-1}$ and if $d(x,y) < \delta_{n-1}$, $x,y \in \widetilde{K}_{n-1}$, then $d(\Psi^{n-1}_{1}(x), \theta^{n-1}_{1}(y)) < \epsilon/4$, where $\theta^{n-1}_s$ is the flow along $W_{n-1}$ for time $s$. If necessary, we shrink $\delta_{n-1}$ further to ensure that the open $\delta_{n-1}$-neighbourhood of $K_{n-1}$ is contained in $\widetilde{K}_{n-1}$.

Now look at the previous subinterval $[n-2,n-1]$. Again we choose $Z_{n-2} \subset \subset \Omega_{n-2}$ containing the flows of all points in $\widetilde{K}_{n-2}$ along $V_{n-2}$ up to time $1$. We choose $\eta_{n-2}, \delta_{n-2} > 0$ so that if $W_{n-2} \in \mathfrak{X}(\Omega_{n-2})$ satisfies $\lVert V_{n-2} - W_{n-2} \rVert_{L^\infty(Z_{n-2})} < \eta_{n-2}$ and if $d(x,y) < \delta_{n-2}$, $x,y \in \widetilde{K}_{n-2}$, then $d(\Psi^{n-2}_{1}(x), \theta^{n-2}_{1}(y)) < \delta_{n-1} / 2$. As before, shrink $\delta_{n-2}$ if necessary to ensure that the open $\delta_{n-2}$-neighbourhood of $K_{n-2}$ is contained in $\widetilde{K}_{n-2}$.

Continuing in this manner we obtain open sets $Z_{j} \subset \subset \Omega_{j}$, and $\eta_{j}, \delta_{j} > 0$ such that conditions as in the previous paragraph hold on each subinterval $[j,j+1]$, $j = 0,\dots,n-1$. Note that on the subinterval $[0, 1]$ we only require that $Z_0 \subset \subset \Omega_0$ contain the flows under $V_{0}$ of points in $K$, and we may also assume that $x = y$.

At each time step $j = 0, \dots, n-1$, we now have the autonomous vector field $V_{j} \in \mathfrak{X}(\Omega_{j})$ together with the open set $Z_{j}\subset\subset\Omega_{j}$ and the number $\eta_{j} > 0$. Since each $\Omega_{j}$ is Runge in $X$ (and recall therefore also Stein) we may approximate $V_{j}$ uniformly on $Z_{j}$ by a global holomorphic vector field $H_j \in \mathfrak{X}(X)$ so that
\[
\lVert V_{j} - H_j \rVert_{L^\infty(Z_{j})} < \eta_{j}/2\,.
\]

Using the fact that $X$ has the density property we can then approximate $H_j$ uniformly on $Z_{j}$ to accuracy $\eta_{j}/2$ by a vector field $W_j \in \mathfrak{X}(X)$ that is in the Lie algebra generated by the complete holomorphic vector fields on $X$. We thus obtain the uniform approximations
\[
\lVert V_{j} - W_j \rVert_{L^\infty(Z_{j})} < \eta_{j},\quad j = 0,\dots,n-1,
\]
that ensure for $x, y \in \widetilde{K}_{j}$, $d(x,y) < \delta_{j}$, we have $d( \Psi^{j}_{1}(x), \theta^j_{1}(y)) < \delta_{j+1}/2$ for $j < n - 1$, and $d( \Psi^{n-1}_{1}(x), \theta^{n-1}_{1}(y)) < \epsilon/4$, where $\theta^j_s$ is the flow along $W_j$ for time $s$. We now apply Corollary~\ref{cor_approxbyauto} to uniformly approximate each $\theta^j_{1}$ on $\widetilde{K}_{j}$ to accuracy $\delta_{j+1}/2$ for $j < n-1$, and to accuracy $\epsilon/4$ when $j = n-1$, by automorphisms $\alpha_{j}$ of $X$. This gives, for $x,y \in \widetilde{K}_{j}$ such that $d(x,y) < \delta_{j}$,
\begin{equation}
\label{eqn_approxPsiGeneral}
d(\Psi^{j}_{1}(x), \alpha_{j}(y)) < \delta_{j+1}\,, \quad j \neq n-1\,,
\end{equation}
and
\begin{equation}
\label{eqn_approxPsiFinal}
d(\Psi^{n-1}_{1}(x), \alpha_{n-1}(y)) < \epsilon/2\,.
\end{equation}

Let $\alpha = \alpha_{n-1} \circ \cdots \circ \alpha_{1} \circ \alpha_0$. Clearly $\alpha$ is an automorphism of $X$. We now show that $\alpha$ approximates $\Phi_n$ uniformly to accuracy $\epsilon$ on $K \subset \Omega$. Let $x \in K$. Then by (\ref{eqn_approxPhi1}) we have $d(\Phi_n(x), A^n_{1}(0,x)) < \epsilon /2$. Letting $x_{1} = \Psi^{0}_{1}(x)$, we have by (\ref{eqn_approxPsiGeneral})
\[
d(x_{1}, \alpha_0(x)) < \delta_{1}\,,
\]
and since $x_{1} \in K_{1}$, we have that $\alpha_0(x) \in \widetilde{K}_{1}$ by our earlier choice of $\delta_1$. Setting $x_{2} = \Psi^{1}_{1}(x_{1})$ and again applying (\ref{eqn_approxPsiGeneral}) we obtain
\[
d(x_{2}, \alpha_{1}\circ\alpha_{0}(x)) < \delta_{2}\,,
\]
and since $x_{2} \in K_{2}$, we have that $\alpha_{1}\circ\alpha_0(x) \in \widetilde{K}_{2}$. Continuing in this way, using (\ref{eqn_approxPsiFinal}) in the final step, we ultimately obtain
\begin{equation}
\label{eqn_approxByAlpha}
d(x_n, \alpha(x)) < \epsilon /2\,,
\end{equation}
where $x_{j} = \Psi^{j-1}_{1}(x_{j-1})$ for each $j = 1,\dots,n$.
Since $x_n = A^n_{1}(0,x)$, combining (\ref{eqn_approxPhi1}) and (\ref{eqn_approxByAlpha}) shows that
\[
d(\Phi_n(x), \alpha(x)) < \epsilon
\]
as required.
\end{proof}

Using this result we can now prove Theorem~\ref{thm_AndersenLempertHolConvex_app}.

\begin{proof}[Proof of Theorem~\ref{thm_AndersenLempertHolConvex_app}]
We will show that $K \subset \Omega$ has a basis of open neighbourhoods $U$ in $\Omega$ such that $\Phi_t(U)$ is Runge in $X$ for each $t \in [0,1]$. Applying Theorem~\ref{AndersenLempert} to one such $U$ then immediately yields the desired conclusion.

Since $K \subset \Omega$ is an $\OO(X)$-convex compact set, there exists a smooth plurisubharmonic exhaustion function $\rho \ge 0$ on $X$ that is strictly plurisubharmonic on $X \setminus K$ and vanishes precisely on $K$ \cite[Thm. 1.3.8]{Stout:2007}. Then each sublevel set $U_\epsilon = \{x \in X : \rho(x) < \epsilon\}$ is Stein and Runge in $X$. We now show that for each $t \in [0,1]$, we may choose $\epsilon > 0$ sufficiently small so that $\Phi_t(U_\epsilon)$ is Runge.

Fix $t\in [0,1]$, and let $V \subset \subset \Omega_t$ be an open set such that $K_t \subset V$. Let $\rho_t = \rho \circ \Phi_t^{-1} \ge 0$. Since $\Phi_t^{-1}$ is holomorphic and injective on $\Omega_t$, $\rho_t$ is strictly plurisubharmonic on $\Omega_t\setminus K_t$ and vanishes precisely on $K_t$. As $K_t$ is $\OO(X)$-convex and compact, by \cite[Thm. 2.6.11]{Hormander:1990} there exists a smooth plurisubharmonic exhaustion function $\widetilde{\tau} \ge 0$ on $X$ that is both strictly positive and strictly plurisubharmonic on $X \setminus V$, and that vanishes on some closed set $V_1 \subset \subset V$ satisfying $K_t \subset V_1^\circ$.

Now let $\chi$ be a smooth cut-off function with values in $[0,1]$ and compact support in $\Omega_t$ that identically equals $1$ on $V$. If $\delta > 0$ is chosen sufficiently small then $\tau_t(x) = \widetilde{\tau}(x) + \delta \chi(x) \rho_t(x)$ is strictly plurisubharmonic on $X \setminus K_t$ and vanishes precisely on $K_t$. Indeed, outside $V$, $\tau_t$ is a small perturbation of $\widetilde{\tau}$ over a compact set and hence is strictly plurisubharmonic, while on $V$ we have $\tau_t = \widetilde{\tau} + \delta \rho_t$, which is strictly plurisubharmonic outside $K_t$. Note that for $x \in V_1$, $\tau_t(x) = \delta \rho_t(x)$.

For $\epsilon > 0$, every sublevel set $\{x \in X : \tau_t(x) < \epsilon\}$ is Runge in $X$. Let $\epsilon_0 > 0 $ be sufficiently small that $\Phi_t(U_{\epsilon_0}) \subset V_1$ and that $\widetilde{\tau}(x) > \delta \epsilon_0$ for $x \in X \setminus V$. Then for $\epsilon < \epsilon_0$,
\[
\Phi_t(U_\epsilon) = \{ x \in V_1 : \rho_t(x) < \epsilon\} = \{ x \in X : \tau_t(x) < \delta \epsilon \}\,,
\]
hence $\Phi_t(U_\epsilon)$ is Runge.

For $t' \in [0,1]$ we write $t' = t + \Delta t$, where $t$ is still fixed as above. Then for all sufficiently small $\lvert \Delta t\rvert$ we have that:
\begin{enumerate}
\item The same set $V$ chosen above for $t$ also satisfies $V \subset \subset \Omega_{t'}$ and $K_{t'}\subset V$.
\item The closed set $V_1$ (on which the  smooth plurisubharmonic function $\widetilde{\tau}$ chosen above vanishes) also satisfies $K_{t'}\subset V_1^\circ$.
\item The support of $\chi$ is compact in $\Omega_{t'}$, and thus $\tau_{t'}(x) = \widetilde{\tau}(x) + \delta \chi(x) \rho_{t'}(x)$ is also strictly plurisubharmonic on $X \setminus K_{t'}$ and vanishes precisely on $K_{t'}$.
\item $\Phi_{t'}(U_{\epsilon_0}) \subset V_1$ for the same $\epsilon_0$ as chosen above.
\end{enumerate}
Thus the same $\epsilon_0$ works for all $t'$ in an open neighbourhood of $t$, and by compactness of $[0,1]$ we can choose $\epsilon_0$ such that $\Phi_t(U_\epsilon)$ is Runge in $X$ for all $\epsilon < \epsilon_0$, for all $t \in [0,1]$.
\end{proof}


\end{document}